\documentclass[oneside,english]{amsart}
\usepackage[T1]{fontenc}
\usepackage[latin9]{inputenc}
\usepackage{amsthm}
\usepackage{amssymb}

\makeatletter

\newtheorem{thm}{Theorem}[section] 
\newtheorem*{thm*}{Theorem}
\newtheorem{lem}[thm]{Lemma}     
\newtheorem{cor}[thm]{Corollary}
\newtheorem*{cor*}{Corollary}
\newtheorem{prop}[thm]{Proposition}
\newtheorem{fact}[thm]{Fact}
 
\newtheorem{defn}[thm]{Definition}
\newtheorem{rem}[thm]{Remark}
\newtheorem{problem}[thm]{Problem}
\newtheorem{example}[thm]{Example}

\newtheorem*{claim*}{Claim}

\usepackage{amssymb}

\begin{document}

\def\Ind#1#2{#1\setbox0=\hbox{$#1x$}\kern\wd0\hbox to 0pt{\hss$#1\mid$\hss}
\lower.9\ht0\hbox to 0pt{\hss$#1\smile$\hss}\kern\wd0}
\def\Notind#1#2{#1\setbox0=\hbox{$#1x$}\kern\wd0\hbox to 0pt{\mathchardef
\nn="3236\hss$#1\nn$\kern1.4\wd0\hss}\hbox to 0pt{\hss$#1\mid$\hss}\lower.9\ht0
\hbox to 0pt{\hss$#1\smile$\hss}\kern\wd0}
\def\ind{\mathop{\mathpalette\Ind{}}}
\def\nind{\mathop{\mathpalette\Notind{}}}

\global\long\def\lstp{\operatorname{Lstp}}

\global\long\def\vstp{\operatorname{vstp}}

\global\long\def\acl{\operatorname{acl}}

\global\long\def\inp{\operatorname{inp}}

\global\long\def\Aut{\operatorname{Aut}}

\global\long\def\M{\operatorname{\mathbb{M}}}

\global\long\def\NTP{\operatorname{NTP}}

\global\long\def\NIP{\operatorname{NIP}}

\global\long\def\IP{\operatorname{IP}}

\global\long\def\TP{\operatorname{TP}}

\global\long\def\NSOP{\operatorname{NSOP}}

\global\long\def\bdn{\operatorname{bdn}}

\global\long\def\tp{\operatorname{tp}}

\global\long\def\qftp{\operatorname{qftp}}

\global\long\def\Q{\operatorname{\mathbb{Q}}}

\global\long\def\Sym{\operatorname{Sym}}

\global\long\def\Stab{\operatorname{Stab}}

\global\long\def\alt{\operatorname{alt}}

\global\long\def\ext{\operatorname{ext}}

\global\long\def\Av{\operatorname{Av}}

\global\long\def\bdd{\operatorname{bdd}}

\global\long\def\mes{\operatorname{\mathfrak{M}}}

\global\long\def\inv{\operatorname{inv}}

\global\long\def\Th{\operatorname{Th}}

\global\long\def\P{\operatorname{\mathbf{P}}}

\global\long\def\bsigma{\operatorname{\boldsymbol{\Sigma}}}

\global\long\def\fs{\operatorname{fs}}
\global\long\def\fsg{\operatorname{fsg}}

\title{External definability and groups in NIP theories}
\author{Artem Chernikov}
\address{\'Equipe de Logique Math\'ematique \\ Institut de Math\'ematiques de Jussieu - Paris Rive Gauche \\ Universit\'e Paris Diderot Paris 7 \\ UFR de Math\'ematiques - case 7012 \\ 75205 Paris Cedex 13 \\ France}
\email{art.chernikov@gmail.com}
\author{Anand Pillay}
\address{Department of Mathematics \\ University of Notre Dame \\ 281 Hurley Hall \\ Notre Dame, IN 46556 \\ US}
\email{apillay@nd.edu}
\author{Pierre Simon}
\address{Univerist\'e de Lyon CNRS \\ Institut Camille Jordan UMR 5208\\ 43 boulevard du 11 novembre 1918 \\ F-69622 Villeurbanne Cedex \\ France}
\email{simon@math.univ-lyon1.fr}
\thanks{The first and the third authors were supported by the European Research Council under the European Unions Seventh Framework Programme (FP7/2007-2013) / ERC Grant agreement no. 291111 and by ANR-13-BS01-0006-01 ValCoMo.\\
The second author was supported by EPSRC grant EP/I002294/1.
}
\maketitle

\begin{abstract} We prove that many properties and invariants of definable groups in $\NIP$
theories, such as definable amenability, $G/G^{00}$, etc., are preserved when passing to the theory of the 
Shelah expansion by externally definable sets, $M^{\ext}$, of a model $M$. In the light of these results we continue the study of the ``definable topological dynamics'' of groups in $\NIP$ theories. In particular we prove the Ellis group conjecture relating the Ellis group to $G/G^{00}$ in some new cases, including definably amenable groups in $o$-minimal structures.
\end{abstract}
\section{Introduction}

Motivation for the work in this paper comes from a new interaction between topological dynamics and model theory, initiated by Newelski \cite{New4} for example. Classical topological dynamics is concerned with understanding a topological group (often discrete) via its continuous actions on compact spaces.  Model theory can provide some new dynamical invariants for discrete groups $G$, which will be explored in future papers. In the current paper
we are concerned rather with new invariants for definable groups suggested by topological dynamics. Given a group $G$ definable in a first order structure $M$, we have the action of $G$ on the Stone space $S_{G}(M)$ of ultrafilters on the Boolean algebra of definable subsets of $G$. $S_{G}(M)$ is a ``tame'' analogue of the Stone-Cech compactification of the discrete group $G$. In analogy with the discrete case, we can study minimal subflows of $S_{G}(M)$ and, under additional assumptions, a corresponding ``Ellis group''.  Newelski and later the second author made some conjectures relating this Ellis group to  a model theoretic invariant $G/G^{00}$ (read in a nonstandard model) of $G$. In the current paper we solve  this conjecture in some important cases.

Let us now describe more of the background behind, and aims of the paper, so as to aid accessibility to a wider audience, although in the body of the paper 
we will freely use reasonably advanced methods from contemporary model theory, with references of course. Model theory studies first order theories.
In the same way as abstract groups are important in mathematics, and algebraic or Lie groups are important in algebraic or differential
geometry, the understanding of groups definable in a given first order theory (or in classes of first order theories) is important for model theory
as well as its applications.  The class of {\em stable} first order theories is at the centre of model theory and  {\em stable group theory} was developed in the 1970's and 80's (see \cite{PoizatStableGroups}), often using terminology (connected components, stabilizers, generics,..) borrowed from the key example of algebraic groups over algebraically closed fields. This general theory applied to other examples such as the theory of differentially closed fields, with a substantial impact on ``differential algebraic groups'' among other topics. Although real Lie groups are outside the domain of stability, they are, more or less, groups definable in $o$-minimal theories, and have been studied by model theorists from this point of view for some time. A common generalization of stable theories and $o$-minimal theories are $\NIP$ theories, characterized by every uniformly definable family of definable sets having finite Vapnik-Chervonenkis dimension. Recently there has been a profitable study of groups definable in $\NIP$ theories, using explicitly techniques and notions from stability theory and stable group theory. See \cite{NIP1,NIP2}. The current paper continues this line of work. In this more general context we have several ``connected components'' of a definable group $G$, $G^{0}$, $G^{00}$, $G^{\infty}$ which coincide in the stable case. Likewise various different notions of ``genericity''. The stable-like $\NIP$ groups are the ``definably amenable'' groups.  As the latter expression suggests, notions of topological dynamics are quite relevant to our study. Topological dynamical notions were  brought
into the picture by Newelski, for example \cite{New4}, and later by Pillay, for example \cite{AnandTopDyn}. 
In \cite{GisPenPil}, Gismatullin, Penazzi and Pillay developed a basic theory built around the notion of a {\em definable} action of a group
$G$ definable in a model $M$, on a compact space $X$, but under a certain assumption on the model $M$, {\em definability of types}. In the stable case this assumption is automatically satisfied, for any model $M$. And in the $\NIP$ case, Shelah has proved  that for a given model $M$ one can ``expand'' it by ``externally definable sets'' to $M^{ext}$ so that $\NIP$ of the new theory $\Th\left(M^{\ext}\right)$ is preserved and the definability of types assumption  is satisfied for $M^{ext}$: 
\begin{fact}
\label{fac: Shelah's theorem on externally definable sets}\cite{ShelahDependentCont} Let $M$
be a model of an $\NIP$ theory $T$. 
\begin{enumerate}
\item The projection of an externally definable subset of $M$ is externally
definable.
\item In particular $\Th\left(M^{\ext}\right)$ eliminates quantifiers,
and is $\NIP$.
\end{enumerate}
\end{fact}

Further study of externally definable sets in $\NIP$ theories, as
well as a refined and uniform version of Shelah's theorem, can be
found in \cite{ExtDefI,2012arXiv1202.2650C}.

So one aim of this paper is to show that many properties of (e.g. definable amenability)  and objects attached to (e.g. $G^{00}$)
a group $G$ definable over a model $M$ of an $\NIP$ theory $T$ are preserved when passing to 
$\Th\left(M^{\ext}\right)$, answering some questions raised in \cite{GisPenPil}. A second aim of this
paper, bearing in mind the above, is to prove some more cases of the ``Ellis group'' conjecture
(originating with Newelski) which says that in the $\NIP$ environment, for suitable groups $G$
definable over a model $M$,  $G/G^{00}$ should coincide with the ``Ellis group'' computed in $\Th\left(M^{\ext}\right)$, where all types over $M^{\ext}$ are definable.  And as is shown in the first part of the paper $G^{00}$ is unchanged when passing to the expanded theory. So the problem is well-defined,
and we answer it in particular for definably amenable groups in $o$-minimal theories, as well as dp-minimal groups. We also study  ``topological dynamical'' properties of groups with ``definable $f$-generics'' (see below), complementing  the study in \cite{AnandTopDyn} of groups with finitely satisfiable generics. 

\medskip
\noindent
Now for some more details.

In Section \ref{sec: Existence of invariant heirs and definability of measures} we establish a couple of general facts about measures in $\NIP$ theories. We show in Theorem \ref{thm: invariant heir of a measure exists} that every measure over a small model in an $\NIP$ theory has a global invariant extension which is also an heir (generalizing the result for types from \cite{CheKap}).
We also observe that the answer to \cite[Question 3.15]{GisPenPil}
is positive in the case of $\NIP$ theories.
\begin{thm*}[\ref{thm: Definability of types implies definabilty of measures}]

\begin{enumerate}
\item Assume that $T$ is $\NIP$, $M\models T$ and all types over $M$
are definable. Then every Borel probability measure on $S\left(M\right)$
is definable (a measure $\mu$ is definable if for every $L$-formula
$\phi\left(x,y\right)$ and closed disjoint subsets $C_{1},C_{2}$
of $\left[0,1\right]$, the sets $\left\{ b\in M:\mu\left(\phi\left(x,b\right)\right)\in C_{1}\right\} $
and $\left\{ b\in M:\mu\left(\phi\left(x,b\right)\right)\in C_{2}\right\} $
are separated by a definable set in $M$).
\item In particular, if $G$ is a definably amenable $M$-definable group,
then it is witnessed by an $M$-definable measure.
\end{enumerate}
\end{thm*}
Examples of structures satisfying the
assumption of the theorem are: any model of a stable theory, $\left(\mathbb{R},+,\cdot\right)$,
$\left(\mathbb{Q}_{p},+.\cdot\right)$, $\left(\mathbb{Z},+,<\right)$
(see \cite[Section 5]{2012arXiv1202.2650C} for a discussion of this
phenomenon).

~

In Section \ref{sec: Lifting measures to Shelah's expansion and preservation of amenability} we study lifting of Keisler measures and related objects to Shelah's expansion.
The main theorem is:
\begin{thm*}[\ref{thm: lifting measures to Shelah's expansion}]
Assume that $T$
is $\NIP$, $M\models T$ and $G$ is an $M$-definable group.
\begin{enumerate}
\item Assume that $G$ is definably amenable,
i.e., there is a Borel probability measure $\mu$ on $S_{G}\left(M\right)$
invariant under the action of $G\left(M\right)$. Then $G$ is still
definably amenable in the sense of $M^{\ext}$: there is some Borel
probability measure $\mu'$ on $S_{G}\left(M^{\ext}\right)$ extending
$\mu$ and $G\left(M\right)$-invariant.
\item Assume that $G$ is definably extremely amenable,
i.e., the action of $G\left(M\right)$ on $S_{G}\left(M\right)$ has
a fixed point $p$. Then $G$ is still definably extremely amenable
in the sense of $M^{\ext}$: there is some $p'\in S_{G}\left(M^{\ext}\right)$
extending $p$ and $G\left(M\right)$-invariant.
\end{enumerate}
\end{thm*}
This answers positively \cite[Question 3.16, (1) and (2)]{GisPenPil}.
We remark that (1) was essentially known for $o$-minimal theories but (2) was
open even in the $o$-minimal case. Our proof combines the existence of invariant heirs for measures from Section \ref{sec: Existence of invariant heirs and definability of measures} (as explained in Section \ref{sec: def amenability and f-generics}) along with the
existence of a canonical continuous retraction from the space of global invariant
measures onto the closed subspace of finitely satisfiable measures (Sections \ref{sec: Extracting the f.s. part}, \ref{sec: Extracting the f.s. part of a measure}).

If $G$ is a group definable over model $M$ of an $\NIP$ theory, then definable amenability of $G$ is equivalent to the existence of a global $f$-generic type of $G$, namely a complete
type $p$ over the monster model $\M$, every left translate $gp$ of which does not fork over $M$  (equivalently is $Aut(\M/M)$-invariant). An $f$-generic type $p$ can fall into one of the two extreme cases: (a) $p$ is $\fsg$ (with respect to $M$), namely every left translate $gp$ is finitely satisfiable in $M$, and (b) $p$ is a definable (over $M$) $f$-generic, namely $p$ is $f$-generic with respect to $M$ and is definable over $M$, \emph{equivalently} every $gp$ is definable over $M$.  So we will observe that both these extreme witnesses of definable amenability are preserved when passing to $\Th\left(M^{\ext}\right)$.

\begin{thm*}[\ref{thm: fsg and definable generics}]
Suppose $T$ is $\NIP$, $M\models T$ and $G$ is a group definable over $M$.
\begin{enumerate}
\item If $G$ has a global $\fsg$ type (with respect to $M$), then $G$ has a global $\fsg$ type with respect to $M^{\ext}$ in  $\Th\left(M^{\ext}\right)$.
\item If $G$ has a global $f$-generic which is definable over $M$, then the same is true for $\Th\left(M^{\ext}\right)$.
\end{enumerate}
\end{thm*}

We also characterize definably extremely amenable groups as those definably amenable groups in which $G^{00} = G$.

~

In Section \ref{sec: connected components} of the paper we study the effect of externally
definable sets on the model-theoretic connected components of definable
groups. Working in a monster model, let $G$ be a definable group
and $A\subseteq\M$. Recall that $G_{A}^{0}$, $G_{A}^{00}$ and $G_{A}^{\infty}$
are defined respectively as the intersection of all subgroups of $G$ of finite
index definable over $A$, the intersection of all subgroups of $G$
of bounded index type-definable over $A$, and the intersection of
all subgroups of $G$ of bounded index invariant over $A$.
The subscript is omitted if $A=\emptyset$. A fundamental fact is:
\begin{fact}
\label{fac: existence of connected components}Let $T$ be $\NIP$
and let $G$ be a definable group.
\begin{enumerate}
\item \cite{MR2361885} $G_{A}^{00}=G_{\emptyset}^{00}$ for every small set $A$, so in particular
the intersection of all type-definable subgroups of bounded index
is type definable over $\emptyset$ and is of index $\leq2^{\left|T\right|}$.
\item  \cite{Sh886} for the abelian case, and
\cite{Jakub} in general: $G_{A}^{\infty}=G_{\emptyset}^{\infty}$ for every small set $A$,
so in particular the intersection of all subgroups of bounded index
invariant over small subsets is invariant over $\emptyset$ and is
of index $\leq2^{\left|T\right|}$.
\end{enumerate}
\end{fact}
Of course $G^{0}\supseteq G^{00}\supseteq G^{\infty}$, and there
are $\NIP$ examples where $G^{00}\supsetneq G^{\infty}$ \cite{AnnalisaAnand}.

So let $M$ be a model of an $\NIP$ theory and let $G\left(M\right)$
be an $M$-definable group; to simplify the notation we assume that
$G$ is the whole universe. Assume that $H$ is an externally definable
subgroup of $M$ (i.e., $H=M\cap\phi\left(x\right)$ for some $\phi\left(x\right)\in L\left(\M\right)$
is a subgroup of $M$). In general it need not contain any $M$-definable
subgroups:
\begin{example}
Let $M\succ\left(\mathbb{R},+,\cdot\right)$ be $\left(2^{\aleph_{0}}\right)^{+}$-saturated.
Then $M$ contains the subgroup $H=\left\{ x\in M:\bigwedge_{r\in\mathbb{R}}\left|x\right|<r\right\} $
of infinitesimal elements. Note that $H$ is externally definable
as ``$M\cap (c<x<d)$'' where $c,d\in\M$ realize the appropriate cuts of
$M$. However $H$ does not contain any $M$-definable subgroups.
\end{example}

However we show that these connected components are not affected by adding
externally definable sets:
\begin{thm*}
Let $T$ be an $\NIP$ theory in the language $L$, and $M\models T$.
Let $T'=\Th\left(M^{\ext}\right)$, and let $\M'$ be a monster model
of $T'$. Let $\M=M'\restriction L$ --- a monster model of $T$. Note that $T'$ is $\NIP$ by Fact \ref{fac: Shelah's theorem on externally definable sets}.
Then we have:
\begin{enumerate}
\item $G^{0}\left(\M\right)=G^{0}\left(\M'\right)$ (Theorem \ref{thm: G^0(M) = G^0(M^ext)}),
\item $G^{00}\left(\M\right)=G^{00}\left(\M'\right)$ (Corollary \ref{cor: G^00 is not changed in Shelah's expansion}),
\item $G^{\infty}\left(\M\right)=G^{\infty}\left(\M'\right)$ (Corollary \ref{cor: G000(M)=G000(Mext)}).
\end{enumerate}
\end{thm*}
For the proof we first establish existence of the corresponding connected
components relatively to a predicate and a sublanguage, and then we
show that each of these relative connected components coincides with the corresponding connected
component of the theory induced on the predicate.
\begin{cor*}
Let $T$ be $\NIP$ and let $M$ be a model of $T$. Assume that $G$ is an
externally definable subgroup of $M$ of finite index. Then it is
internally definable. 
\end{cor*}

Finally in Section \ref{sec: TopDyn and Ellis} of the paper we return to the motivating context of ``tame topological dynamics''. We will explain the set-up in some detail in Section \ref{sec: TopDyn and Ellis}. But we give here a brief description of the notions so as to be able to state the main results. The context here is an $\NIP$ theory $T$, model $M$ of $T$ and definable group $G$, defined over $M$. The type space $S_{G}(M^{\ext})$ is acted on ``definably'' by $G(M)= G(M^{\ext})$, 
and also has a canonical semigroup structure, continuous in the first coordinate. There exist minimal closed $G$-invariant subsets of the type space, and elements of such ``minimal subflows'' are called {\em almost periodic types}. Any two minimal closed $G(M)$-invariant subspaces of $S_{G}(M^{\ext})$ are ``isomorphic'' and coincide with the unique universal minimal definable $G(M)$-flow.  So this is a rather basic invariant of $G$ (or rather $G(M)$) given by the set-up of topological dynamics. In \cite{AnandTopDyn} it was proved that when $G$ has $\fsg$ then there is a unique minimal closed $G(M)$-invariant subspace of $S_{G}(M^{\ext})$ and it coincides with the set of generic types. We will study the other extreme  case of definable amenability, when $G$ has an $f$-generic type, definable over $M$. And we prove as a part of Proposition \ref{prop: definable f-generics}:

\begin{thm*}
Suppose $G$ has a global $f$-generic, definable over $M$. Then for $p\in S_{G}(M^{\ext})$, $p$ is almost periodic if and only if $p$ is a ``definable $f$-generic'' in the sense that the unique global heir of $p$ is $f$-generic.
\end{thm*}

As proved in the previous section, $G/G^{00}$ is the same whether computed in $T$ or in $Th(M^{\ext})$ and we just write $G/G^{00}$. 
Fix a minimal closed $G(M)$-invariant subspace $\mathcal M$ of $S_{G}(M^{\ext})$, and an idempotent $u\in \mathcal{M}$. Then $u{\mathcal M}$ is a subgroup of $S_{G}(M^{\ext})$, which we call the Ellis group  (attached to $M,G$) and whose isomorphism type does not depend on the choice of $\mathcal M$ or $u$. In fact there is also a certain non-Hausdorff topology on $u{\mathcal M}$, but it will not concern us in the current paper. The canonical surjective homomorphism $G\to G/G^{00}$ factors naturally through the space $S_{G}(M^{\ext})$, namely we have a well-defined continuous surjection $\pi:S_{G}(M^{\ext})\to G/G^{00}$ taking $\tp(g/M)$ to the coset $gG^{00}$, and the restriction of $\pi$ to the group $u{\mathcal M}$ is a surjective homomorphism. 
We will say that the Ellis group $u{\mathcal M}$ coincides with (or equals) $G/G^{00}$ if $\pi|u{\mathcal M}$ is an isomorphism. It was suggested by Newelski in \cite{New4} that in many cases, $u{\mathcal M}$ does equal $G/G^{00}$. Let us formalize this by conjecturing that when $G$ is definably amenable, then $u{\mathcal M}$ equals $G/G^{00}$.  This was essentially proved in \cite{AnandTopDyn} when $G$ is $\fsg$. We will prove some more cases in Sections \ref{sec: EllisStarts} -- \ref{EllisEnds}:

\begin{thm} 
\label{thm: Ellis group conjectures}
Under the assumptions above, the Ellis group coincides with $G/G^{00}$ in the following cases:
\begin{enumerate}
\item \label{thm: EllisDefExtrAm} $G$ is definably extremely amenable.
\item \label{thm: EllisFSG} $G$ is $\fsg$.
\item \label{thm: EllisDefFGen} $G$ has a definable $f$-generic with respect to $M$.
\item  \label{thm: EllisDPMin} $G$ is definably amenable, $dp$-minimal.
\item \label{thm: EllisOMin} $T$ is an $o$-minimal expansion of a real closed field and $G$ is definably amenable  (computed in $T$ or equivalently in $Th(M^{\ext})$). 
\end{enumerate}

\end{thm}

\subsection*{Notation}

$T$ will always denote a complete $\NIP$ theory in a language $L$ and $\M \models T$ will be a monster model. As usual, $S\left(A\right)$ denotes the space of types over $A$.
We also write $S^{\inv}\left(A,B\right)$ (resp. $S^{\fs}\left(A,B\right)$)
for the set of types over $A$ invariant (resp. finitely satisfiable
) over $B$. Both are closed subsets of $S\left(A\right)$ as $S^{\fs}\left(A,B\right)$
is the closure of the set of types over $A$ realized in $B$ and
\[
S^{\inv}\left(A,B\right)=\left\{ p\in S\left(A\right):\bigwedge_{a\equiv_{B}a',\phi\left(x,y\right)\in L\left(B\right)}\left(p\vdash\phi\left(x,a\right)\leftrightarrow\phi\left(x,a'\right)\right)\right\} \mbox{.}
\]

Whenever we say measure over a set of parameters $A$, we mean a finitely additive Keisler measure
(equivalently, a regular Borel probability measure on $S\left(A\right)$), see e.g. \cite{NIP2}. For a measure $\mu$ we denote by $S(\mu)$ its support: the set of types weakly random for $\mu$, i.e. the closed set of all $p$ such that for any $\phi(x)$, $\phi(x) \in p$ implies $\mu ( \phi(x) ) >0$. Let $\mes(A)$ denote the set of measures over $A$, it is naturally equipped with a compact topology as a closed subset of $[0,1]^{L(A)}$ with the product topology. Every type over $A$ can be identified with the $\{0,1\}$-measure concentrating on it, thus $S(A)$ is identified with a closed subset of $\mes(A)$.

We will assume some basic knowledge of forking for types and measures. E.g., in $\NIP$ a type does not fork over a
model if and only if it is invariant over it, etc \cite{NIP2,CheKap}.

\subsection*{Acknowledgements}
We are grateful to the referee for a very thorough reading and for pointing out numerous deficiencies in the original version of the article.

\section{Existence of invariant heirs and definability of measures} \label{sec: Existence of invariant heirs and definability of measures}

\subsection{Existence of global invariant heirs for measures over models}

We will be assuming that $T$ is a complete $\NIP$ theory throughout the article.

\begin{fact}[\cite{CheKap}]
\label{fac: forking for types in NIP}Let $T$ be $\NIP$ and $M\models T$.
\begin{enumerate}
\item \label{fac: forking for types in NIP 1}A formula $\phi\left(x,a\right) \in L(\M)$ forks over $M$ if and only if it
divides over $M$.
\item \label{fac: forking for types in NIP 2}Every type $p\left(x\right)\in S\left(M\right)$, where $x$ can be
a tuple of variables of arbitrary length, admits a global extension
which is both $M$-invariant and an heir over $M$.
\end{enumerate}
\end{fact}
The aim of this section is to generalize \ref{fac: forking for types in NIP 2} to arbitrary measures
in $\NIP$ theories, i.e. to demonstrate that every measure over a
model of an $\NIP$ theory admits a global invariant heir. 
\begin{defn}\label{def: heir of a measure}
We say that $\nu\in\mes\left(\M\right)$ is an \emph{heir} of $\mu\in\mes\left(M\right)$
if for any finitely many formulas $\phi_{0}\left(x,a\right),\ldots,\phi_{n}\left(x,a\right)\in L\left(\M\right)$
and $r_{0},\ldots,r_{n}\in\left[0,1\right)$,
if $\bigwedge_{i\leq n}\left(\nu\left(\phi_{i}\left(x,a\right)\right)>r_{i}\right)$
then $\bigwedge_{i\leq n}\left(\mu\left(\phi_{i}\left(x,b\right)\right)>r_{i}\right)$
for some $b\in M$.\end{defn}
\begin{rem}

\begin{enumerate}
\item Note that for types we recover the usual notion of an heir.
\item A weaker notion of an heir of a measure was defined in \cite[Remark 2.7]{NIP1},
but working with that definition does not seem sufficient for our
purposes.
\end{enumerate}
\end{rem}
Given $p_0, \ldots, p_{n-1} \in S(A)$ and $\phi(x,a) \in L(A)$, we set $\Av \left( p_0, \ldots, p_{n-1}; \phi(x,a) \right) = \frac{\left|\left\{ i<n:\phi\left(x,a\right)\in p_{i}\right\} \right|}{n}$. Then one defines  $\mu = \Av(p_0, \ldots, p_{n-1}) \in \mes(A)$, the average measure of $p_0, \ldots, p_{n-1}$, by setting $\mu \left( \phi(x,a) \right) = \Av \left( p_0, \ldots, p_{n-1}; \phi(x,a) \right)$ for all $\phi(x,a) \in L(A)$.

The following is a corollary of the VC-theorem from \cite[Section 4]{NIP2}.
It is stated there for a single formula, but easily generalizes to
a finite set of formulas by encoding them into one.
\begin{fact}
\label{fac: measure is average of types}Let $\mu$
a measure on $S\left(A\right)$, $\Delta\left(x\right)=\left\{ \phi_{i}\left(x,y_{i}\right)\right\} _{i<m}$
a finite set of $L$-formulas, and let $\varepsilon>0$ be arbitrary.
Then there are some types $p_{0},\ldots,p_{n-1}\in S\left(A\right)$
such that for every $a\in A$ and $\phi\left(x,y\right)\in\Delta$,
we have 
\[
\left|\mu\left(\phi\left(x,a\right)\right)-\Av\left(p_{0},\ldots,p_{n-1};\phi\left(x,a\right)\right)\right|\leq\varepsilon\mbox{.}
\]
Furthermore, we may assume that $p_{i}\in S\left(\mu\right)$, the
support of $\mu$, for all $i<n$.\end{fact}
\begin{thm}\label{thm: invariant heir of a measure exists}
Let $M$ be a model of $T$. Then every $\mu\in\mes\left(M\right)$
has a global extension $\nu\in\mes\left(\M\right)$ which is both
invariant over $M$ and an heir of $\mu$.\end{thm}
\begin{proof}
Let $\mu\in\mes\left(M\right)$ and $\varepsilon \in [0,1)$ be given, and let $\Delta$ be a finite
set of $L(M)$-formulas. Let $H_{\mu,\Delta,\varepsilon}$ be the set of global $\Delta$-heirs
of $\mu$ up to an $\varepsilon$-mistake and let $I$ is the set of $M$-invariant
global measures:
\begin{align*}
   I =  \left\{ \nu \in \mes(\M) :  \nu\left(\phi\left(x,a\right)\right)=\nu\left(\phi\left(x,b\right)\right) \mbox{ for all } a\equiv_{M}b\in\M,\phi\left(x,y\right)\in L\left(M\right)\right\} \mbox{,}
\end{align*}
\begin{align*}
    H_{\mu,\Delta,\varepsilon} = \left\{ \nu \in \mes(\M) :  \bigvee_{\phi\in\Delta} \left( \nu\left(\phi\left(x,a\right)\right)\leq r_{\phi} \right) \mbox{ for all } a \in \M \mbox{ and all } (r_{\phi})_{\phi \in \Delta} \in \left[0,1\right)^{|\Delta|} \right.\\
    \left. \mbox{ such that } \bigwedge_{\phi \in \Delta} \left( \mu(\phi(x,b)) > r_{\phi}-\varepsilon \right) \mbox{ does not hold for any }b \in M \right\} \mbox{.}
\end{align*}


Let also $H_{\mu}$ be the set of global heirs of
$\mu$. Note that all these sets are closed in $\mes\left(\M\right)$ and 
that $H_{\mu}=\bigcap_{\Delta\subseteq L(M)\mbox{ finite},n\in\omega}H_{\mu,\Delta,\frac{1}{n}}$ (if $\nu$ belongs to the set on the right and $\bigwedge_{i<n} \nu(\phi_i(x,a)) > r_i$, then there is some $\varepsilon > 0$ such that $\bigwedge_{i<n} \left( \nu(\phi_i(x,a)) > r_i + \varepsilon \right)$, and since $\nu \in H_{\mu, \{\phi_0, \ldots, \phi_{n-1}\}, \varepsilon}$ we find some $b \in M$ satisfying $\bigwedge_{i<n} \mu(\phi_i(x,b)) > r_i$).

Fix an arbitrary $\varepsilon>0$ and finite $\Delta\left(x\right)\subseteq L(M)$.
By Fact \ref{fac: measure is average of types} there are some $p_{0},\ldots,p_{n-1}\in S\left(M\right)$
such that $\left|\mu\left(\phi\left(x,a\right)\right)-\Av\left(p_{0},\ldots,p_{n-1};\phi\left(x,a\right)\right)\right|\leq\varepsilon$
for all $a\in M$ and $\phi\left(x,y\right)\in\Delta$. Let $p\left(x_{0},\ldots,x_{n-1}\right)\in S_{n}\left(M\right)$
be some completion of $p_{0}\left(x_{0}\right)\cup\ldots\cup p_{n-1}\left(x_{n-1}\right)$,
then by Fact \ref{fac: forking for types in NIP}\ref{fac: forking for types in NIP 2} there is some
$q\left(x_{0},\ldots,x_{n-1}\right)\in S_{n}\left(\M\right)$ ---
a global $M$-invariant heir of $p\left(x_{0},\ldots,x_{n-1}\right)$.
Let $q_{i}=q\restriction x_{i}\in S\left(\M\right)$ for $i<n$, and
let $\nu_{\varepsilon,\Delta}=\Av\left(q_{0},\ldots,q_{n-1}\right)\in\mes\left(\M\right)$.
\begin{claim*}
$\nu_{\varepsilon,\Delta}\in I$, i.e. it is invariant over $M$.\end{claim*}
\begin{proof}
By $\NIP$ it is enough to show that $\nu$ does not fork over $M$. If
it does then $\nu\left(\phi\left(x,a\right)\right)>0$ for some $\phi\left(x,a\right)$
forking over $M$, which by the definition of $\nu_{\varepsilon,\Delta}$
implies that $\phi\left(x,a\right)\in q_{i}$ for some $i<n$ ---
a contradiction.\end{proof}
\begin{claim*}
$\nu_{\varepsilon,\Delta}\in H_{\mu,\Delta,\varepsilon}$.\end{claim*}
\begin{proof}
Assume that $\bigwedge_{\phi\in\Delta}\left(\nu_{\varepsilon,\Delta}\left(\phi\left(x,a\right)\right)>r_{\phi}\right)$ holds
for some $a\in\M$ and $r_{\phi}\in\left[0,1\right), \phi \in \Delta$. Then $\bigwedge_{\phi\in\Delta}\left(\frac{\left|\left\{ i<n:\phi\left(x,a\right)\in q_{i}\right\} \right|}{n}>r_{\phi}\right)$,
that is $\bigwedge_{\phi\in\Delta}\left(\frac{\left|\left\{ i<n:\phi\left(x_{i},a\right)\in q\right\} \right|}{n}>r_{\phi}\right)$.
As $q$ is an heir over $M$, there is some $b\in M$
satisfying $\bigwedge_{\phi\in\Delta}\left(\frac{\left|\left\{ i<n:\phi\left(x_{i},b\right)\in p\right\} \right|}{n}>r_{\phi}\right)$,
so $\bigwedge_{\phi\in\Delta}\left(\frac{\left|\left\{ i<n:\phi\left(x,b\right)\in p_{i}\right\} \right|}{n}>r_{\phi}\right)$.
But then by the choice of $p_{i}$'s it follows that $\mu\left(\phi\left(x,b\right)\right)>r_{\phi}-\varepsilon$
for every $\phi\in\Delta$, as wanted.
\end{proof}
Finally, assume towards a contradiction that $H_{\mu}\cap I$ is empty.
But then it follows by compactness of $\mes\left(\M\right)$ that
$H_{\mu,\Delta,\frac{1}{n}}\cap I$ is empty for some finite $\Delta\subseteq L(M)$
and $n\in\omega$. However, $\nu_{\frac{1}{n},\Delta}\in H_{\mu,\Delta,\frac{1}{n}}\cap I$
by the previous claims.\end{proof}

\begin{example}
Every $M$-definable measure $\mu \in \mes(\M)$ is an invariant heir over $M$.
\end{example}

\subsection{Definability of types implies definability of measures}
We give another application of Fact \ref{fac: measure is average of types} and show that if all types over a model are definable, then measures over it are definable as well.

\begin{thm}
\label{thm: Definability of types implies definabilty of measures}
Assume that $M\models T$ and that all types over $M$
are definable. Then:
\begin{enumerate}
\item Every Borel probability measure on $S\left(M\right)$
is definable.
\item In particular, if $G$ is a definably amenable $M$-definable group,
then it is witnessed by an $M$-definable measure.
\end{enumerate}
\end{thm}
\begin{proof}
We are assuming that all types over $M$ are definable, and let
$\mu$ be a measure on $S\left(M\right)$. We want to show that $\mu$
is definable. Let $\phi\left(x,y\right)$ and $C_{1},C_{2}$ closed
disjoint subsets of $\left[0,1\right]$ be given. It then follows
that there is some $\varepsilon>0$ such that no point of $C_{1}$ has
any point of $C_{2}$ in its $\varepsilon$-neighbourhood. Let $D_{i}=\left\{ b\in M:\mu\left(\phi\left(x,b\right)\right)\in C_{i}\right\} $
for $i\in\left\{ 0,1\right\} $.

Let $p_{0},\ldots,p_{n-1}\in S\left(M\right)$ by as given by Fact
\ref{fac: measure is average of types} for $\phi$, $\mu$ and $\frac{\varepsilon}{2}$.
As each of $p_{i}$'s is definable, there is some $d_{\phi}p_{i}\left(y\right)\in L\left(M\right)$
such that for any $a\in M$ we have $\phi\left(x,a\right)\in p_{i}\Leftrightarrow M\models d_{\phi}p_{i}\left(a\right)$.
Note that $\Av\left(p_{0},\ldots,p_{n-1};\phi\left(x,a\right)\right)$
can only take values from the finite set $\left\{ \frac{m}{n}:m<n\right\} $,
and let $C$ be the set of those values whose distance from $C_{1}$
is less that $\frac{\varepsilon}{2}$. Let $D=\left\{ a\in M:\Av\left(p_{0},\ldots,p_{n-1};\phi\left(x,a\right)\right)\in C\right\} $,
it is definable by some boolean combination of $d_{\phi}p_{i}\left(y\right)$'s,
so $L\left(M\right)$-definable. It is then easy to see from the definition
that $D_{1}\subseteq D$ and that $D\cap D_{2}=\emptyset$, as wanted.
\end{proof}

\section{Lifting measures to Shelah's expansion and preservation of amenability} \label{sec: Lifting measures to Shelah's expansion and preservation of amenability}

\subsection{Definable amenability and $f$-generic types}\label{sec: def amenability and f-generics}

\begin{defn}
A definable group $G$ is \emph{definably amenable}
if there is a left-invariant finitely additive probability measure defined on the algebra of all definable subsets of $G$. 
\end{defn}
First we summarize some known facts about definably amenable groups in $\NIP$
theories which will be used freely later on in the text.

\begin{defn}
A global type $p \in S_G(\M)$ is \emph{left $f$-generic} over a small model $M$ if $g \cdot p$ does not fork over $M$ for all $g \in G$ (equivalently, $g \cdot p$ is invariant over $M$ for all $g \in G$).
\end{defn}

\begin{fact}\label{fac: properties of def am groups}
\begin{enumerate}
\item \cite[5.10,5.11]{NIP2} $G$ is definably amenable if and only if for some (equivalently,
any) small model $M$, there is a global type $p \in S(\M)$ which is left $f$-generic over $M$.
\item \cite[Section 5]{NIP1} Definable amenability is a property of the theory: If $S_G(M)$ admits a $G$-invariant measure and $M \equiv N$, then $S_G(N)$ admits a $G$-invariant measure.
\item \cite[5.6(i)]{NIP2}  If $p\in S\left(\M\right)$ is $f$-generic then $\Stab\left(p\right)=G^{00}=G^{\infty}$, where $\Stab(p) = \{ g \in G : g \cdot p = p \}$.
\end{enumerate}
\end{fact}

Next we consider extending $G$-invariant ($M$-invariant) measures to larger sets of parameters.
\begin{prop}\label{prop: heir remains G-invariant}
Assume that $\mu\in\mes\left(M\right)$ is $G\left(M\right)$-invariant
and $\nu\in\mes\left(\M\right)$ is an heir of $\mu$ (in the sense of Definition \ref{def: heir of a measure}). Then $\nu$
is $G\left(\M\right)$-invariant.\end{prop}
\begin{proof}
Assume not, then for some $\phi\left(x,a\right)\in L\left(\M\right)$ and $g\in G\left(\M\right)$ we have  $\nu\left(\phi\left(x,a\right)\right) = r_1, \nu\left(\phi\left(g^{-1} \cdot x,a\right)\right) = r_2$ and $r_1 \neq r_2$, say $r_1 > r_2$. Then, taking $r = (r_1+r_2)/2$, we have $\nu\left(\phi\left(x,a\right)\right)>r \land \nu\left(\neg\phi\left(g^{-1} \cdot x,a\right)\right)>1-r \land \nu( g \in G) >0$. As $\nu$ is an heir of $\mu$, this implies
that $\mu\left(\phi\left(x,b\right)\right)>r \land \mu\left(\neg\phi\left(h^{-1} \cdot x,b\right)\right)>1-r \land \mu (h \in G)>0$
for some $b\in M$ and $h\in M$, which implies that $h \in G(M)$ and contradicts $G\left(M\right)$-invariance
of $\mu$. 
\end{proof}

\begin{prop}\label{prop: extending to a G-inv M-inv measure}
If $M\models T$ and $\mu \in \mes(M)$ is $G(M)$-invariant,  then there is some $\nu \in \mes(\M)$ extending $\mu$, which is both $G(\M)$-invariant and $M$-invariant. \end{prop}
\begin{proof}
By Theorem \ref{thm: invariant heir of a measure exists}, $\mu$ admits a global $M$-invariant heir $\nu$. By Proposition  \ref{prop: heir remains G-invariant} $\nu$ is $G$-invariant.\end{proof}

Finally for this section, we characterize definable extreme amenability.
\begin{defn}
A definable group $G$ is \emph{definably extremely amenable}
if there is a $G$-invariant type $p \in S_G(\M)$.
\end{defn}
It is easy to see that definable extreme amenability is a property of the theory: by compactness, if there is a $G(M)$-invariant type in $S_G(M)$ and $M \equiv N$, then there is a $G(N)$-invariant type in $S_G(N)$.

\begin{prop}\label{prop: def ext am iff} An NIP group $G$ is definably extremely amenable if and only if it is definably
amenable and $G=G^{00}$. 
\end{prop}
\begin{proof}
If $G$ is definably amenable, then there is an $f$-generic $p$
such that $\Stab\left(p\right)=G^{00}=G$. But then $p$ is $G$-invariant,
so $G$ is definably extremely amenable.\\
Conversely, as $G$ is definably extremely amenable, for some small
model $M$ there is some $p\in S\left(M\right)$ which is $G\left(M\right)$-invariant.
Let $p^{*}\in S\left(\M\right)$ be a non-forking heir of $p$. It
follows that $p^{*}$ is $G\left(\M\right)$-invariant and $M$-invariant,
so in particular $f$-generic over $M$. But then $G^{00}=\Stab\left(p^{*}\right)=G$.
\end{proof}
\begin{rem}
In particular, if $T$ is stable then $G$ is definably extremely
amenable if and only if $G=G^{0}$. However, in the $\NIP$ case definable
amenability does not follow even from $G=G^{\infty}$. Indeed, given
a saturated real closed field $K$, $G=SL\left(2,K\right)$ is simple
as an abstract group modulo its finite center. Then $G=G^{\infty}$,
but this group is not definably amenable.\end{rem}

\subsection{\label{sec: Extracting the f.s. part}Extracting the finitely satisfiable
part of an invariant type}

We present a construction due to the third author from \cite{SimContraction}.

Let $M\models T$ and $N\succ M$ be $\left|M\right|^{+}$-saturated.
We let $M^{\ext}$ be a Shelah's expansion of $M$ in the language
$L'=\left\{ R_{\phi}\left(x\right):\phi\left(x\right)\in L\left(N\right)\right\} $
with $R_{\phi}\left(M\right)=M\cap\phi\left(x\right)$,
$T'=\Th_{L'}\left(M^{\ext}\right)$. Let $\left(N',M',\left(R_{\phi}\right)_{\phi\in L\left(N\right)}\right)$
be an $\left|N\right|^{+}$-saturated expansion of $\left(N,M,\left(R_{\phi}\right)_{\phi\in L\left(N\right)}\right)$
with a new predicate $\P\left(x\right)$ naming $M$. It follows that
$M'\succ^{L'}M$ and that still $R_{\phi}\left(M'\right)=M'\cap\phi\left(x\right)$.
We can identify $M'\restriction L$ with the monster model $\M$ of $T$. 
\begin{prop}
\label{prop: honest definitions}Working in $T'$, for every $L$-type
$p\in S^{\inv}_L\left(M',M\right)$ and $R_{\phi}(x)\in L'$,
if $p\left(x\right)\cup R_{\phi}\left(x\right)$ is
consistent then $p\left(x\right)\vdash R_{\phi}\left(x\right)$
(and in fact $p|_{M^{*}}\vdash R_{\phi}\left(x\right)$ for any $\left|N\right|^{+}$-saturated
$M\prec M^{*}\prec M'$).\end{prop}
\begin{proof}
In the proof of the existence of honest definitions \cite[Proposition 1.1]{ExtDefI}, it is shown that if $\phi(x) \land p(x) \land \P(x)$ is consistent then $p_{0}\left(x\right)\land \P\left(x\right)\vdash\phi\left(x\right)$
for some small $p_{0}\subseteq p$, which translates to $p\left(x\right)\vdash R_{\phi\left(x\right)}\left(x\right)$ in view of the previous paragraph.
\end{proof}
Given $p\in S^{\inv}_L\left(M',M\right)$ we define $p'=\left\{ R_{\phi\left(x\right)}\left(x\right):p\vdash R_{\phi\left(x\right)}\left(x\right)\right\} $.
It is clearly a complete $L'$-type over $M^{\ext}$ and does not
depend on the choice of $N$ as it was only used to define the language.
Thus we can identify it with a global $L$-type $F_{M}\left(p\right)=\left\{ \phi\left(x\right)\in L\left(M'\right):\phi\left(M\right)=R_{\psi\left(x\right)}(M) \mbox{ for some } \psi(x) \in L(N) \mbox{ with } R_{\psi}(x) \in p'  \right\} $
finitely satisfiable in $M$.

Recall that given a global type $p\left(x\right)$ and a definable
function $f$, one defines $f_{*}\left(p\right)=\left\{ \phi\left(x\right):\phi\left(f\left(x\right)\right)\in p\right\} $.
If $p$ is $M$-invariant and $f$ is $M$-definable, then $f_{*}\left(p\right)$
is also $M$-invariant.
\begin{prop}
\label{prop: properties of F on types}The map $F_{M}$ satisfies
the following properties:
\begin{enumerate}
\item \label{prop: properties of F on types1} $F_{M}\left(p\right)|_{M}=p|_{M}$.
\item \label{prop: properties of F on types2} $F_{M}$ is a continuous retraction from $S^{\inv}\left(\M,M\right)$
onto $S^{\fs}\left(\M,M\right)$.
\item \label{prop: properties of F on types3} If $f$ is an $M$-definable function, then $f_{*}\left(F_{M}\left(p\right)\right)=F_{M}\left(f_{*}\left(p\right)\right)$.
\end{enumerate}
\end{prop}
\begin{proof}
\ref{prop: properties of F on types1} Clear from the construction.

\ref{prop: properties of F on types2} Fix a formula $\phi(x) \in L(M')$, and let $\psi(x) \in L(N)$ be such that $\phi(M) = \psi(M)$. Then, unwinding the definition, we see that $\phi(x) \in F_M(p) \Leftrightarrow p(x) \land \P(x) \vdash \psi(x)$. Thus $F_{M}^{-1}\left(\phi\left(x\right)\right)=\bigcup\left\{ \chi(x) \in L(M') :  \chi(x) \land \P(x) \vdash  \psi(x) \right\}$, and $F_M$ is continuous.

Now assume that $p$ is actually finitely satisfiable in $M$, and
that $\phi\left(x\right)\in L\left(\M\right)$ is such that $\phi\left(x\right)\in p$
and $\neg\phi\left(x\right)\in F_{M}\left(p\right)$. But then $\neg\phi\left(M\right)=R_{\psi\left(x\right)}\left(M\right)$ and $R_\psi(x)\in p'$.
This means that there is some $\chi\left(x\right)\in p$ such that
$\chi\left(x\right)\vdash R_{\psi}\left(x\right)$.
But as $\chi\left(x\right)\land\phi\left(x\right)\in p$, by
finite satisfiability there is some $a\in M$ with $a\models R_{\psi\left(x\right)}\left(x\right)\land\phi\left(x\right)$
--- a contradiction. Thus $F_{M}$ is the identity on $S^{\fs}\left(\M,M\right)$.

\ref{prop: properties of F on types3} First observe that it is enough to show that $f_{*}\left(p'\right)=\left(f_{*}\left(p\right)\right)'$.
By compactness and Proposition \ref{prop: honest definitions} there
is some $M\subseteq B\subseteq M'$ such that $\left|B\right|=\left|N\right|$
and $p|_{B}\vdash p'$, $f_{*}\left(p\right)|_{B}\vdash\left(f_{*}\left(p\right)\right)'$.
Let $a$ in $M'$ realize $p|_{B}$, and let $b=f\left(a\right)$.
Then $b\models f_{*}\left(p\right)|_{B}$, thus $ $$b\models\left(f_{*}\left(p\right)\right)'$.
On the other hand, as $a\models p'$, it follows that $b\models f_{*}\left(p'\right)$.
\end{proof}

In fact, \cite[Proposition 1.1]{ExtDefI}
implies the following more explicit statement:
\begin{prop}
Assume that $\left(N',M'\right)$ is a sufficiently saturated elementary extension of $\left(N,M\right)$. Then
for every $\phi\left(x\right)\in L\left(N\right)$ there are some
$\psi\left(x\right),\psi'\left(x\right)\in L\left(M'\right)$ such
that $\psi\left(M' \right)\subseteq\phi\left(M' \right)\subseteq\psi'\left(M' \right)$, $\psi(M)=\psi'(M)=\phi(M)$
and $\psi'\left(x\right)\setminus\psi\left(x\right)$ divides over
$M$ (in the sense of $T$).\end{prop}
\begin{proof}
Let $\phi(x)$ be given. The proposition gives us a formula $\psi(x) \in L(M')$ such that $\psi(M') \subseteq \phi(M')$ and moreover no $M$-invariant type in $S_L(M')$ satisfies $\phi(x) \setminus \psi(x)$.
Applying the proposition again to $\neg \phi(x)$, we find some $\chi(x) \in L(M')$ such that $\chi(M') \subseteq \neg \phi(M')$ and no $M$-invariant type in $S_L(M')$ satisfies $\neg \phi(x) \setminus \chi(x)$. But then take $\psi'(x) = \neg \chi(x)$. It follows that $\psi(M') \subseteq \phi(M') \subseteq \psi' (M')$ and that no $M$-invariant type in $S_L(M')$ satisfies $\psi'(x) \setminus \psi(x)$. By saturation of $M'$, $\NIP$ and Fact \ref{fac: forking for types in NIP}\ref{fac: forking for types in NIP 1} it follows that $\psi'(x) \setminus \psi(x)$ divides over $M$.

\end{proof}

\subsection{\label{sec: Extracting the f.s. part of a measure}Extracting the finitely satisfiable
part of an invariant measure}
Now we extend this map $F_M$ to measures.
\begin{rem}
A measure $\mu\in\mes\left(A\right)$ is invariant (finitely satisfiable)
over $B\subseteq A$ if and only if every $p\in S\left(\mu\right)$
is invariant (finitely satisfiable) over $B$.\end{rem}
\begin{proof}
It is clear that if $\mu$ is invariant (finitely satisfiable) over
$B$ then every $p\in S\left(\mu\right)$ is invariant (finitely satisfiable)
over $B$. Conversely, assume that $\mu\left(\phi\left(x,a\right)\right)>0$.
Then it is easy to see by compactness that there is some $p\in S\left(\mu\right)$
with $\phi\left(x,a\right)\in p$.
\end{proof}
Assume that $\mu\in\mes\left(\M\right)$ is $M$-invariant. Then we
can define a measure $\mu'$ on $S^{\inv}\left(\M,M\right)$ by setting
$\mu'\left(S^{\inv}\left(\M,M\right)\cap\phi\left(x,a\right)\right)=\mu\left(\phi\left(x,a\right)\right)/\mu\left(S^{\inv}\left(\M,M\right)\right)$.
If $\mu\left(\phi\left(x,a\right)\bigtriangleup\psi\left(x,b\right)\right)>0$
then there is some $p\in S\left(\mu\right)$ with $\phi\left(x,a\right)\bigtriangleup\psi\left(x,b\right)\in p$.
By the previous remark $p$ is $M$-invariant, thus $\phi\left(x,a\right)\cap S^{\inv}\left(\M,M\right)\neq\psi\left(x,b\right)\cap S^{\inv}\left(\M,M\right)$,
which implies that $\mu'$ is a well-defined.

Conversely, given a measure $\mu'$ on $S^{\inv}\left(\M,M\right)$
we define $$\mu\left(\phi\left(x,a\right)\right)=\mu'\left(\phi\left(x,a\right)\cap S^{\inv}\left(\M,M\right)\right).$$
Then $\mu$ is a measure on $S\left(\M\right)$, and every type in
the support of $\mu$ is invariant, thus $\mu$ is invariant.
\begin{rem}
\label{rem: invariant measure is a measure on invariant types}An $M$-invariant (resp. finitely satisfiable) measure $\mu\in\mes\left(\M\right)$
is the same thing as a measure on $S^{\inv}\left(\M,M\right)$ (resp.
$S^{\fs}\left(\M,M\right)$).\end{rem}
\begin{defn}
\label{fac: push-forward measure}Let $\left(X_{1},\Sigma_{1}\right)$,
$\left(X_{2},\Sigma_{2}\right)$ be measurable spaces and let a Borel
mapping $f:X_{1}\to X_{2}$ be given (e.g. a continuous map). Then,
given a measure $\mu:\Sigma_{1}\to\left[0,1\right]$, the pushforward
of $\mu$ is defined to be the measure $f_{*}\left(\mu\right):\Sigma_{2}\to\left[0,1\right]$
given by $\left(f_{*}\left(\mu\right)\right)\left(A\right)=\mu\left(f^{-1}\left(A\right)\right)$
for $A\in\Sigma_{2}$.
\end{defn}
Given an $M$-invariant global measure $\mu$, by Remark \ref{rem: invariant measure is a measure on invariant types}
we view it as a measure $\mu'$ on the space of invariant types $S^{\inv}\left(\M,M\right)$.
By continuity of $F_{M}$
we thus get a push-forward measure $\left(F_{M}\right)_{*}\left(\mu'\right)$
on the space $S^{\fs}\left(\M,M\right)$. Again by Remark \ref{rem: invariant measure is a measure on invariant types}
this determines a measure $\nu$ on $S\left(\M\right)$ which is finitely
satisfiable in $M$. We define $F_{M}\left(\mu\right)=\nu$.
\begin{prop}
\label{prop: properties of F on measures}The map $F_{M}$ satisfies
the following properties:
\begin{enumerate}
\item \label{prop: properties of F on measures 1} $F_{M}\left(\mu\right)|_{M}=\mu|_{M}$.
\item \label{prop: properties of F on measures 2} $F_{M}$ is a continuous retraction from $\mes^{\inv}\left(\M,M\right)$
to $\mes^{\fs}\left(\M,M\right)$.
\item \label{prop: properties of F on measures 3} If $f$ is an $M$-definable function, then $f_{*}\left(F_{M}\left(\mu\right)\right)=F_{M}\left(f_{*}\left(\mu\right)\right)$.
\end{enumerate}
\end{prop}
\begin{proof}
Follows from Proposition \ref{prop: properties of F on types} by
unwinding the definition of $F_{M}\left(\mu\right)$.
\end{proof}

\subsection{Lifting measures to Shelah's expansion}

The following fact is well-known for types, and we observe that it
easily generalizes to measures.
\begin{prop}
\label{prop: measures in M^ext are f.s.} The
measures on $M^{\ext}$ are in a natural one-to-one correspondence
with global measures finitely satisfiable in $M$.\end{prop}
\begin{proof}
By quantifier elimination, every definable subset of $M^{\ext}$ is
of the form $\phi\left(M,a\right)$ for some $a\in\M$. Given a global
measure $\mu$ finitely satisfiable in $M$, we define a measure $\mu'\in\mes\left(M^{\ext}\right)$
as follows: given an externally definable set $X\subseteq M$, we set
$\mu'\left(X\right)=\mu\left(\phi\left(x,a\right)\right)$ for some
$\phi\left(x,a\right)\in L\left(\M\right)$ such that $X=\phi\left(M,a\right)$.
It is well-defined because if $X=\phi\left(M,a\right)=\psi\left(M,b\right)$
then $\mu\left(\phi\left(x,a\right)\right)=\mu\left(\psi\left(x,b\right)\right)$
(as otherwise $\mu\left(\phi\left(x,a\right)\bigtriangleup\psi\left(x,b\right)\right)>0$,
thus there is some $c\models\phi\left(x,a\right)\bigtriangleup\psi\left(x,b\right)$
in $M$ by finite satisfiability --- a contradiction) and is clearly
a measure on $S\left(M^{\ext}\right)$.

Conversly, given a measure $\mu'\in\mes\left(M^{\ext}\right)$, for
$\phi\left(x,a\right)\in L\left(\M\right)$ we define $\mu\left(\phi\left(x,a\right)\right)=\mu'\left(\phi\left(M,a\right)\right)$.
It is easy to see that $\mu$ is a global measure and that whenever
$\mu\left(\phi\left(x,a\right)\right)>0$ then $\mu'\left(\phi\left(M,a\right)\right)>0$,
thus $\phi\left(M,a\right)$ is non-empty.
\end{proof}

We are ready to prove the main theorem of the section.
\begin{thm}
\label{thm: lifting measures to Shelah's expansion}Assume that $M\models T$ and $G$ is an $M$-definable group.
\begin{enumerate}
\item Let $\mu$ be a $G\left(M\right)$-invariant measure on $S_{G}\left(M\right)$. Then there is some measure $\mu'$ on $S_{G}\left(M^{\ext}\right)$ which extends
$\mu$ and is $G\left(M\right)$-invariant.
\item Assume that the action of $G\left(M\right)$ on $S_{G}\left(M\right)$ has
a fixed point $p$. Then there is some $p'\in S_{G}\left(M^{\ext}\right)$ which extends $p$ and is $G\left(M\right)$-invariant.
\end{enumerate}
\end{thm}
\begin{proof}
Let $\mu\in\mes\left(M\right)$ be a $G\left(M\right)$-invariant
measure on $S_{G}\left(M\right)$. By Proposition \ref{prop: extending to a G-inv M-inv measure}
there is some global measure $\mu'$ invariant over $M$ which is
in addition $G\left(\M\right)$-invariant. Now let $\nu=F_{M}\left(\mu'\right)$
be a global measure finitely satisfiable in $M$, as constructed in
Section \ref{sec: Extracting the f.s. part}. By Proposition \ref{prop: properties of F on measures}\ref{prop: properties of F on measures 1}
it is still a measure on $S_{G}\left(\M\right)$, extending $\mu$.
We claim that $\nu$ is $G\left(M\right)$-invariant. Indeed, by Proposition
\ref{prop: properties of F on measures}\ref{prop: properties of F on measures 3} and $G\left(M\right)$-invariance
of $\mu'$, for every $g\in G\left(M\right)$ we have $g\cdot\nu=g\cdot F_{M}\left(\mu'\right)=F_{M}\left(g\cdot\mu'\right)=F_{M}\left(\mu'\right)=\nu$.
So $\nu$ is $G\left(M\right)$-invariant and finitely satisfiable
in $M$, thus by Proposition \ref{prop: measures in M^ext are f.s.}
it corresponds to a $G\left(M\right)$-invariant measure on $S_{G}\left(M^{\ext}\right)$,
as wanted.

For the case of the existence of a fixed point in $S_{G}\left(M\right)$
the proof goes through by restricting to zero-one measures.
\end{proof}
We remark that as both existence of a fixed point and definable amenability
are properties of the theory, the same holds in the monster model
of $\Th\left(M^{\ext}\right)$.

\subsection{Definable $f$-generics and fsg}

We aim towards proving Theorem \ref{thm: fsg and definable generics}. As before, we are assuming that $T$ has $\NIP$ and $M\models T$. We begin by pointing out that any definable complete type over $M$ has a unique extension to a complete type over $M^{\ext}$. This was observed in Claim 1, Proposition 57, of \cite{2012arXiv1202.2650C}, but we give another proof here. We will use the notation at the beginning of Subsection 2.2, namely $M, N, M', N',\P, L'$. In particular $M$ as an $L'$-structure is precisely $M^{\ext}$, and $M'$ as an $L'$-structure is a saturated model of $Th(M^{\ext})$, whose $L$-reduct can be identified with the monster model of $T$. 

\begin{lem} \label{lem: extensions of definable types}
Suppose $p(x)\in S(M)$ is definable. Then $p(x)$ implies a unique complete type $p^{*}(x)\in S(M^{\ext})$. Moreover if 
$\bar p$ is the unique heir of $p$ over the $L$-structure $M'$, then again $\bar p$ implies a unique complete type over $M'$ as an $L'$-structure, which is precisely the unique heir of $p^{*}$.
\end{lem}

\begin{proof} Let $\bar p$  be the unique heir of $p$ over $M'$ (as an $L$-structure). By Proposition \ref{prop: honest definitions}, $\bar p$ implies a unique complete type $p^{*}(x)$ over $M^{\ext}$. So if $R_{\phi}$ is in $p^{*}(x)$, then  in  a saturated elementary extension of $\left(N',M',\left(R_{\phi}\right)_{\phi\in L\left(N\right)}\right)$ we have the implication ${\bar p}(x)\wedge \P(x) \vdash  R_{\phi}(x)$, so by compactness there is $\psi(x,c)\in {\bar p}$ such that 
$\left(N',M',\left(R_{\phi}\right)_{\phi\in L\left(N\right)}\right) \models \forall x\in \P(\psi(x,c) \rightarrow R_{\phi}(x))$
\newline
Let $\chi(y)$ be an $L$-formula over $M$ which is the $\psi(x,y)$-definition of $\bar p$ (equivalently of $p$). 
Hence $\models \chi(c)$, so by Tarski-Vaught, there is $c_{0}\in M$ such that $\left(N,M,\left(R_{\phi}\right)_{\phi\in L\left(N\right)}\right)\models \chi(c_{0}) \wedge \forall x\in \P(\psi(x,c_{0}) \rightarrow R_{\phi}(x))$

As $\psi(x,c_{0})\in p(x)$, we see that $p(x)$ implies $R_{\phi}(x)$ as required.

This proves the first part of the Lemma. 

The moreover clause follows in a similar fashion. Namely by the first part, $\overline p$, being definable, implies a unique complete type over $(M')^{\ext}$, in particular implies a unique complete $L'$-type over $M'$, which can be checked to be the unique heir of $p^{*}$. 
\end{proof}

\begin{thm}
\label{thm: fsg and definable generics}
Suppose that $M\models T$ and $G$ is a group definable over $M$.
\begin{enumerate}
\item \label{thm: fsg and definable generics 1} If $G$ has a global $\fsg$ type (with respect to $M$), then $G$ has a global $\fsg$ type with respect to $M^{\ext}$ in  $\Th\left(M^{\ext}\right)$.
\item \label{thm: fsg and definable generics 2} If $G$ has a global $f$-generic which is definable over $M$, then the same is true for $\Th\left(M^{\ext}\right)$.
\end{enumerate}
\end{thm}
\begin{proof}   
\ref{thm: fsg and definable generics 1} Let $L''$ be the language of $(M')^{\ext}$, and let $M''$ be a saturated elementary extension of $M'$ as an $L''$-structure. As $G$ is $\fsg$ in $T$ and $M'' \succ_L M$, there is some $p \in S_L(M'')$ such that $gp$ is finitely satisfiable in $M$ for all $g \in G(M'')$. It determines a complete type $q \in S_{L''}((M')^{\ext})$ such that moreover $g q$ is finitely satisfiable in $M$ for all $g \in G(M')$. Let $r = q \restriction L'$, it satisfies the same property. As $M' \succ_{L'} M^{\ext}$ is a saturated extension, it follows that $G$ is $\fsg$ in $\Th(M^{\ext})$.

\ref{thm: fsg and definable generics 2} We continue with the same notation.  Our assumptions give us a complete $L$-type $\bar p$ over $M'$,  which is definable over $M$ and such that 
for every $g\in G(M')$, $g{\bar p}$ is definable over $M$. So the stabilizer of ${\bar p}$ is $G^{00}(M')$. By  Lemma \ref{lem: extensions of definable types}, 
${\bar p}$ extends to a unique complete $L'$-type ${\bar p}^{*}$ over $M'$ which is moreover definable over $M$. So $\Stab({\bar 
p}^{*})$ is also $G^{00}(M')$, in particular has bounded index, so clearly ${\bar p}^{*}$ is also a global $f$-generic of $G$, definable over 
$M^{\ext}$  in $Th(M^{\ext})$, as required.
\end{proof}

\section{Connected components}\label{sec: connected components}

In this section we will show that the model-theoretic connected components
are not affected by adding externally definable sets. For simplicity
of notations we will be assuming that our group $G$ is the whole universe.

~

Let $N\succeq M\models T$. By an elementary pair of models $\left(N,M\right)$
we always mean a structure in the language $L_{\P}=L\cup\left\{ \P\left(x\right)\right\} $
whose universe is $N$ and such that $\P\left(N\right)=M$. We say
that an $L_{\P}$-formula is \emph{bounded} if it is of the form $Q_{0}x_{0}\in\P\ldots Q_{n-1}x_{n-1}\in\P\phi\left(x_{0},\ldots,x_{n-1},\bar{y}\right)$
where $Q_{i}\in\left\{ \exists,\forall\right\} $ and $\phi\left(\bar{x},\bar{y}\right)\in L$.
We will denote the set of all bounded formulas by $L_{\P}^{\bdd}$.

An $L_{\P}$-formula $\phi\left(x,y\right)\in L_{\P}$ is $\NIP$
over $\P$ (modulo some fixed theory of elementary pairs $T_{\P}$)
if for some $n<\omega$ there are no $\left(b_{i}:i<n\right)$ in
$N$ and $\left(a_{s}:s\subseteq n\right)$ in $\P$ such that $\phi\left(a_{s},b_{i}\right)$
$\Leftrightarrow$ $i\in s$. By the usual compactness argument, $\phi\left(x,y\right)$
is not $\NIP$ over $\P$ if and only if there is some $\left(N,M\right)\models T_{\P}$
in which we can find an $L_{\P}$-indiscernible sequence $\left(a_{i}:i<\omega\right)$
in $\P$ and $b\in N$ such that $\left(N,M\right)\models\phi\left(a_{i},b\right)$
$\Leftrightarrow$ $i$ is even (any sufficiently saturated pair would
do).
\begin{rem}
\label{rem: bounded formulas are NIP}Let $\left(N,M\right)$ be an
elementary pair of models of an $\NIP$ theory $T$. Then every bounded
formula is $\NIP$ over $\P$ modulo $T_{\P}=\Th\left(N,M\right)$.\end{rem}
\begin{proof}
Let $\phi\left(x,y\right)\in L_{\P}^{\bdd}$ be given, and assume
that it is not $\NIP$ over $\P$. By the previous paragraph this
means that there is some $\left(N,M\right)\models T_{\P}$, $\left(a_{i}:i<\omega\right)$
in $M$ and $b\in N$ such that $\left(N,M\right)\models\phi\left(a_{i},b\right)$
$\Leftrightarrow$ $i$ is even. Take some $M'\succ M$ which is $|M|^{+}$-saturated. By Shelah's
theorem $\Th\left(M^{\ext}\right)$ eliminates quantifiers, that is
for every $a\in N$ there is some $\psi\left(x,y\right)\in L$ and
$c \in M'$ such that $\phi\left(M,a\right)=\psi\left(M,c\right)$.
In particular $M'\models\psi\left(a_{i},c\right)$ $\Leftrightarrow$
$i$ is even, contradicting the assumption that $T=\Th_{L}\left(M'\right)$
is $\NIP$.
\end{proof}
Note that the theory $T_{\P}$ of pairs need not be $\NIP$ in general.
In \cite[Section 2]{ExtDefI} it is shown that if every $L_{\P}$
formula is equivalent to a bounded one, then $T_{\P}$ is $\NIP$.

\subsection{$G^{0}$}

We begin with the easiest case. Let $N\succ M$ be saturated, of size
bigger than $\left|2^{M}\right|^{+}$.

First we generalize some basic $\NIP$ lemmas to the case of externally
definable subgroups.
\begin{lem}
For any formula $\phi\left(x,y\right)$ and integer $n$ there is
some $k$ such that:\\
there are $b_{0},\ldots,b_{k-1}\in N$ such that $\phi\left(M,b_{i}\right)$
is a subgroup of index $\leq n$ for each $i<k$ and for any $b\in N$,
if $\phi\left(M,b\right)$ is a subgroup of index $\leq n$ then $\bigcap_{i<k}\phi\left(M,b_{i}\right)\subseteq\phi\left(M,b\right)$. \end{lem}
\begin{proof}
The usual proof of the Baldwin-Saxl lemma goes through showing
that there is some $k$ such that any finite intersection $\bigcap_{i<m}\phi\left(M,b_{i}\right)$
is equal to a subintersection of size $\leq k$. As there are at most
$\left|2^{M}\right|$ different externally definable subgroups of
$M$, by compactness and saturation of $N$ we can thus find $b_{0},\ldots,b_{k-1}\in N$
as wanted.\end{proof}
\begin{defn}
Let $G_{\phi,n}^{M}$ be the
intersection of all externally definable subgroups of $M$ of index
$\leq n$, of the form $\phi\left(M,b\right)$ for some external parameter
$b$. 
\end{defn}
By the previous lemma it follows that there is some $k$ and
$\left(b_{i}:i<k\right)$ from $N$ such that $G_{\phi,n}^{M}=\bigcap_{i<k}\phi\left(M,b_{i}\right)$.
\begin{lem}
\label{lem: properties of G_phi,n^M}
\begin{enumerate}
\item \label{lem: properties of G_phi,n^M 1} $G_{\phi,n}^{M}$ is $\Aut\left(M\right)$-invariant.
\item \label{lem: properties of G_phi,n^M 2} $G_{\phi,n}^{M}$ and $G_{\phi,n}^{M'}$ have the same index for any
two saturated models $M,M'$ (and it is bounded by $kn$).
\end{enumerate}
\end{lem}
\begin{proof}
\ref{lem: properties of G_phi,n^M 1} Note that every $\sigma\in\Aut\left(M\right)$ extends to an automorphism
$\sigma'$ of the pair $\left(N,M\right)$ (indeed, $\sigma$ is a
partial automorphism of $N$, thus extends to an automorphism $\sigma'$
of $N$, which in particular fixes $M$ setwise). We conclude as $a\in\phi\left(M,b\right)\Leftrightarrow\sigma\left(a\right)\in\phi\left(M,\sigma'\left(b\right)\right)$
and $\phi\left(M,\sigma'\left(b\right)\right)$ is still of finite
index in $M$.

\ref{lem: properties of G_phi,n^M 2} The previous lemma gives the same upper bound $k$ for any model
$M$ as it only depends on the VC-dimension of $\phi$ in models of
$T$, hence the bound on the index.

Second, note that if $M\prec M'$ then the index can only go up. We
show that it doesn't. Let $\psi\left(M',b\right)$ be an externally
definable subgroup of $M'$ of bounded index, say of index $l$. Then
we add a new predicate $R$ naming it. We see that $M'\models\left\{ \forall x_{0},\ldots,x_{n}\in R\forall x_{0}',\ldots,x_{n}'\notin R\exists y\,\bigwedge_{i<n}\left(\psi\left(x_{i},y\right)\land\neg\psi\left(x_{i}',y\right)\right)\right\} _{n\in\omega} \\
\cup\left\{ R\mbox{ is a subgroup of index }l\right\} $.
By resplendence we can expand $M$ to a model of the same sentences.
It thus follows by compactness that $R$ is an externally $\psi$-definable
subgroup of $M$ of index $l$. Now applying this observation to $\psi\left(x,y_{0},\ldots,y_{k-1}\right)=\bigwedge_{i<k}\phi\left(x,y_{i}\right)$
and $l=kn$ we can conclude.\end{proof}
\begin{thm} \label{cor: ext def fin ind is def}
\label{thm: G^0(M) = G^0(M^ext)} Let $M\models T$ be arbitrary. Then any externally definable subgroup
of $M$ of finite index is definable. In particular $G^{0}\left(M\right)=G^{0}\left(M^{\ext}\right)$.\end{thm}
\begin{proof}
Assume first that $M$ is saturated. So we have $M\prec N$ and $G_{\phi,n}^{M}=\bigcap_{i<k}\phi\left(M,b_{i}\right)$
with $b_{i}\in N$ for $i<N$. Let $\left(N',M'\right)\succ\left(N,M\right)$
be a saturated extension of the pair. Observe that $\bigcap_{i<k}\phi\left(M',b_{i}\right)$
has the same index in $M'$ as $\bigcap_{i<k}\phi\left(M,b_{i}\right)$
in $M$ by elementarity. On the other hand, $G_{\phi,n}^{M'}$ has
the same index in $M'$ as $G_{\phi,n}^{M}$ in $M$ by Lemma \ref{lem: properties of G_phi,n^M}.
Thus $ $$\bigcap_{i<k}\phi\left(M',b_{i}\right)\supseteq G_{\phi,n}^{M'}$
and their indexes are equal. This implies that $\bigcap_{i<k}\phi\left(M',b_{i}\right)=G_{\phi,n}^{M'}$.
But as $G_{\phi,n}^{M'}$ is $\Aut_{L}\left(M'\right)$-invariant
(by Lemma \ref{lem: properties of G_phi,n^M}) and definable in a
saturated structure $ $$\left(N',M'\right)$, it follows by compactness
that it is actually definable in $M'$, and thus again by elementarity
$G_{\phi,n}^{M}$ is definable in $M$. As $G^{0}\left(M^{\ext}\right)=\bigcap_{n<\omega,\phi\in L}G_{\phi,n}^{M}$
and every $G_{\phi,n}^{M}$ is of finite index, it follows that $G^{0}\left(M^{\ext}\right)=G^{0}\left(M\right)$.

Let now $M$ be arbitrary, and let $\phi(M,b)$ be an externally definable subgroup of finite index, definable over some $N \succ M$. Let $(N',M') \succ (N,M)$ be a saturated extension, by elementarity $\phi(M',b)$ is an externally definable subgroup of $M'$ of finite index. By the previous paragraph, as $M'$ is saturated, it then contains some $M'$-definable subgroup of finite index $\psi(M',c)$, hence $\phi(M',b) = \bigcup_{i<n'} (g_i \cdot \psi(M',c))$ for some $g_0,\ldots,g_{n'-1} \in M'$. By elementarity of the extension it follows that $\phi(M,b) = \bigcup_{i<n'} (g'_i \cdot \psi(M,c'))$ for some $g'_0,\ldots,g'_{n'-1},c' \in M$, in particular it is $M$-definable.
\end{proof}

\subsection{Type-definable groups, invariant groups and bounded index in non-saturated
models.}

In our proofs for relative $G^{00}$ and $G^{\infty}$ we will be
using non-saturated models, so we prefer to make it precise what we
will mean by ``type-definable'', ``invariant'' and ``bounded
index'' etc. in this situation. Recall that we are also assuming that $G(\M) = \M$ for simplicity of notation.
\begin{defn}
\label{def: hereditary subgroup}Let $(N,M)$ be an elementary pair, and let $\bsigma\left(x\right)$
be a disjunction of complete $L$-types over a subset of $N$, each of which is consistent with $\P(x)$ (modulo $T_{\P}= \Th(N,M))$. 
\begin{enumerate}
\item We say that $\bsigma\left(x\right)$ is a \emph{hereditary subgroup}
of $\P(x)$ if $\Sigma\left(M' \right)$
is a subgroup of $M'$ for every $(N',M') \succ (N,M)$.
\item If $\bsigma\left(x\right)$ is a hereditary subgroup of $\P(x)$,
we say that it is of \emph{hereditarily bounded index} if for every
saturated $(N',M') \succ (N,M)$ the group $\bsigma\left( M' \right)$ is of bounded
index in $M'$, i.e. the index is less than the saturation
of $(N',M')$.
\end{enumerate}
\end{defn}
The following two lemmas are rather standard.
\begin{lem}
\label{lem: sugbroups are hereditary in a saturated model}Let $(N,M)$ be an elementary pair,
and let $\bsigma\left(x\right)$ be a disjunction of complete
$L$-types over a set $A\subseteq N$, each of which is consistent with $\P(x)$.
\begin{enumerate}
\item \label{HerLocalVer} Assume that:

\begin{enumerate}
\item For every $p\left(x\right),q\left(x\right)\in\bsigma\left(x\right)$
and every sequence of formulas $\bar{\phi}=\left(\phi_{r}\left(x\right)\right)_{r\left(x\right)\in\bsigma}$
with $\phi_{r}\left(x\right)\in r\left(x\right)$ there are some $\psi_{p}\left(x\right)\in p,\psi_{q}\left(x\right)\in q,n\in\omega$
and $r_{0},\ldots,r_{n-1}$ such that $\psi_{p}\left(x\right)\land\psi_{q}\left(y\right)\land \P \left(x\right)\land \P \left(y\right)\rightarrow\bigvee_{i<n}\phi_{r_{i}}\left(x\cdot y\right)\land \P \left(x \cdot y\right)$.

\item \label{HerGlobalVer}For every $p\left(x\right)\in\bsigma\left(x\right)$ the uniquely
determined $L$-type $p\left(x^{-1}\right)$ is also in $\bsigma\left(x\right)$.
\end{enumerate}

Then $\bsigma\left(M\right)$ is a subgroup of $M$.

\item If $(N,M)$ is saturated and $A$ is a small subset of $N$, then the
converse holds.
\end{enumerate}
\end{lem}

\begin{proof}
It is straightforward to check that \ref{HerLocalVer}(a) implies that $a,b \models \bsigma(x) \Rightarrow a \cdot b \models \bsigma(x) $, and that \ref{HerLocalVer}(b) implies $a \models \bsigma(x) \Rightarrow a^{-1} \models \bsigma(x)$ for all $a,b \in M$, and the converse follows by compactness and saturation of $(N,M)$.
\end{proof}

\begin{lem}
\label{lem: bounded index is hereditary in a saturated model}
Assume that $(N,M)$ is saturated with $|N| = \bar{\kappa}$, that $\bsigma\left(x\right)$ is a
small disjunction of complete $L$-types over a small set $A \subset N$ compatible with $\P(x)$, with $\bar{\kappa} \gg |A|$, and that
$\bsigma\left(M\right)$ is a subgroup of $M$. Then
the following are equivalent:
\begin{enumerate}
\item \label{HerIndLoc} For every sequence of formulas $\bar{\phi}=\left(\phi_{p}\right)_{p\in\bsigma}$
with $\phi_{p}\left(x\right)\in p$, there are some $\phi_{0},\ldots,\phi_{n-1}\in\bar{\phi}$
and $m\in\omega$ such that for any pairwise different $\left(a_{i}\right)_{i<m}$
from $M$ we have $a_{i}^{-1}a_{j}\models\bigvee_{k<n}\phi_{k}\left(x\right)$
for some $i<j<m$.
\item \label{HerIndGlob} $\bsigma\left(M\right)$ is of bounded index in $M$,
that is $\left[M:\bsigma\left(M\right)\right]<\bar{\kappa}$.
\end{enumerate}
\end{lem}
\begin{proof}
\ref{HerIndLoc} $\Rightarrow$ \ref{HerIndGlob}: Assume that the index of $\bsigma\left(M\right)$ is unbounded,
and let $\kappa=\left|\bsigma\right|$. Let $\lambda=\left(2^{2^{\left(\left|A\right|^{\kappa}\right)}}\right)^{+}$,
then we can find $\bar{a}=\left(a_{i}\right)_{i\in\lambda}$ in $M$
such that $a_{i}^{-1}a_{j}\notin\bsigma\left(M\right)$ for all $i<j<\lambda$.
That is, for every $p\in\bsigma$ there is some $\phi_{i,j}^{p}\in p$
such that $a_{i}^{-1}a_{j}\models\neg\phi_{i,j}^{p}$. Let $\bar{\phi}_{i,j}=\left(\phi_{i,j}^{p}\right)_{p\in\bsigma}$.
As there are only $\left|A\right|^{\kappa}$ possible values of $\bar{\phi}_{i,j}$,
by Erd\H os-Rado there is some infinite $I\subseteq\lambda$ such
that $\bar{\phi}_{i,j}=\bar{\phi}$ for all $i<j\in I$. But then
$a_{i}^{-1}a_{j}\models\bigwedge_{p\in\bsigma}\neg\phi_{p}$ for all
$i<j\in I$, so \ref{HerIndLoc} fails for $\bar{\phi}$.

\ref{HerIndGlob} $\Rightarrow$ \ref{HerIndLoc}: Assume that \ref{HerIndLoc} fails for some $\bar{\phi}$,
that is for every $\phi_{0},\ldots,\phi_{n-1}\in\bar{\phi}$ and every
$m\in\omega$ we can find some $\left(a_{i}\right)_{i<m}$ in $M$
such that $a_{i}^{-1}a_{j}\models\bigwedge_{k<n}\neg\phi_{k}\left(x\right)$
for all $i<j<m$. It follows by compactness and saturation of $(N,M)$ that for
any $\kappa<\bar{\kappa}$ we can find some $\left(a_{i}\right)_{i<\kappa}$
in $M$ such that $a_{i}^{-1}a_{j}\models\bigwedge_{p\in\bsigma}\neg\phi_{p}\left(x\right)$
for all $i<j<\kappa$. So $a_{i}^{-1}a_{j}\notin\bsigma\left(M\right)$,
and the index of $\bsigma\left(M\right)$ is unbounded.\end{proof}

\begin{rem}
\label{rem: finitary characterization of hereditary subgroup and hereditarily bounded index}Note
that if $(N',M') \succ (N,M)$ and $\bsigma\left(x\right)$ consists of
$L$-types over $N$ consistent with $\P(x)$, then Lemma \ref{lem: sugbroups are hereditary in a saturated model}\ref{HerLocalVer}
and Lemma \ref{lem: bounded index is hereditary in a saturated model}\ref{HerIndLoc}
hold in $(N',M')$ if and only if they hold in $(N,M)$. This means that:
\begin{itemize}
\item $\bsigma\left(x\right)$ is a hereditary subgroup of $\P \left(x \right)$
if and only if it satisfies Lemma \ref{lem: sugbroups are hereditary in a saturated model}\ref{HerLocalVer}
in $(N,M)$.
\item $\bsigma\left(x\right)$ has hereditarily bounded index in $\P \left(x\right)$
if and only if it satisfies Lemma \ref{lem: bounded index is hereditary in a saturated model}\ref{HerIndLoc}
in $(N,M)$.
\end{itemize}
\end{rem}
\begin{lem}
\label{lem: L-type over P determines hereditary properties in the pair} Let $\bsigma\left(x,\bar{y}\right)$
be a collection of complete $L$-types each of which is compatible with $\P(x)$.
\begin{enumerate}
\item \label{lem: L-type over P determines hereditary properties in the pair 1} Assume we are given elementary pairs $\left(N_{i},M_{i}\right)$ and
$\bar{b}_{i}\in N_{i}$ for $i\in\left\{ 0,1\right\} $. Assume that
$\tp_{L_{\P}^{\forall,\bdd}}\left(\bar{b}_{0}\right)=\tp_{L_{\P}^{\forall,\bdd}}\left(\bar{b}_{1}\right)$
(that is, they agree on all formulas of the form $\forall z_{0}\ldots z_{n}\in\P\phi\left(z_{0},\ldots,z_{n},\bar{y}\right)$
with $\phi\in L$, in the corresponding pairs). Assume that $\bsigma\left(x,\bar{b}_{0}\right)$
is a hereditary subgroup of $\P(x)$ of hereditarily bounded index
in the pair $\left(N_{0},M_{0}\right)$. Then $\bsigma\left(x,\bar{b}_{1}\right)$
is a hereditary subgroup of $\P(x)$ of hereditarily bounded index
in the pair $\left(N_{1},M_{1}\right)$.
\item \label{lem: L-type over P determines hereditary properties in the pair 2}Assume that we are given elementary
pairs $\left(N_{i},M_{i}\right)$ such that $\left(M_{i}:i<\kappa\right)$
and $\left(N_{i}:i<\kappa\right)$ are $L$-elementary chains. Assume that $\bar{b} \in N_0$ is such that:

\begin{enumerate}

\item $\tp_{L_{\P}^{\forall,\bdd}}\left(\bar{b}\right)$ evaluated in $(N_i,M_i)$
is constant for all $i$,
\item $\bsigma\left(x,\bar{b}\right)$ is a hereditary subgroup
of $\P(x)$ of hereditarily bounded index in the pair $\left(N_{0},M_{0}\right)$. 
\end{enumerate}

Let $M=\bigcup_{i<\kappa}M_{i}, N=\bigcup_{i<\kappa}N_{i}$. Then $\tp_{L_{\P}^{\forall,\bdd}}\left(\bar{b}\right)$ in $(N,M)$ is the same as in $(N_0, M_0)$, in particular $\bsigma\left(x,\bar{b}\right)$ is a hereditary subgroup of $\P(x)$
of hereditarily bounded index in the pair $\left(N,M\right)$.

\end{enumerate}
\end{lem}
\begin{proof}
\ref{lem: L-type over P determines hereditary properties in the pair 1} In view of the Remark \ref{rem: finitary characterization of hereditary subgroup and hereditarily bounded index}
we have to check that Lemma \ref{lem: sugbroups are hereditary in a saturated model}\ref{HerLocalVer}
and Lemma \ref{lem: bounded index is hereditary in a saturated model}\ref{HerIndLoc}
hold for $\bsigma\left(x,\bar{b}_{1}\right)$ in $\left(N_{1},M_{1}\right)$.
But this follows directly from $\tp_{L_{\P}^{\forall,\bdd}}\left(\bar{b}_{0}\right)=\tp_{L_{\P}^{\forall,\bdd}}\left(\bar{b}_{1}\right)$
and the assumption on $\bsigma\left(x,\bar{b}_{0}\right)$.

\ref{lem: L-type over P determines hereditary properties in the pair 2} In view of \ref{lem: L-type over P determines hereditary properties in the pair 1} and the assumption it is enough to show that $\tp_{L_{\P}^{\forall,\bdd}}\left(\bar{b}\right)$
is the same in $(N,M)$ as in some/any $(N_i,M_i)$ for $i<\kappa$.
Let $\psi\left(\bar{y}\right)=\forall z_{0}\ldots z_{n-1} \in \P \phi\left(z_{0},\ldots,z_{n-1},\bar{y}\right)\in L_{\P}^{\forall,\bdd}$
be given. Assume that $\left(N,M\right)\models\neg\phi\left(a_{0},\ldots,a_{n-1},\bar{b}\right)$
with $\phi\in L$ and $a_{i}$ from $M=\P\left(N\right)$. It follows
by construction that there is some $\alpha<\kappa$ such that $a_{0},\ldots,a_{n-1}$
are in $ $$M_{\alpha}$. As $N \succ_L N_{\alpha}$,
we have that $\left(N_{\alpha},M_{\alpha}\right)\models\neg\phi\left(a_{0},\ldots,a_{n-1},\bar{b}\right)$, i.e. $(N_\alpha, M_\alpha) \models \neg \psi (\bar{b})$. And the converse is clear.
\end{proof}

\subsection{$G^{00}$}

Let $\left(N,M\right)$ be a saturated elementary pair (of models of an $\NIP$ theory $T$, as before).
\begin{defn}
Let $L'$ be a collection of $L_{\P}$-formulas such that $L\subseteq L'\subseteq L_{\P}$. We consider all subgroups of $\P\left(N\right)=M$ of bounded index
(that is of index less than the saturation) and definable as $\bsigma\left(M,B\right)$
where $B$ is a small tuple from $N$ and $\bsigma$ is a partial
$L'$-type over $B$. Let
$G_{L'\left(B\right)}^{00}\left(N,M\right)$ be defined as the intersection
of all such groups, and let $G_{L'}^{00}\left(N,M\right)=\bigcap_{B\subset N,\mbox{small}}G_{L'\left(B\right)}^{00}\left(N,M\right)$.
\end{defn}

The following is standard.
\begin{fact}
\label{fac: type-def and invariant implies type-definable}Let $\left(N,M\right)$
be a saturated pair, and let $\bsigma\left(x,\bar{y}\right)$ be a
partial type. Let $\bar{b}$ be from $N$, and assume that $\bsigma\left(M,\bar{b}\right)$
is $\Aut_{L}\left(M\right)$-invariant. Then $\bsigma\left(M,\bar{b}\right)$
is definable by a partial $L$-type over $\emptyset$.\end{fact}
\begin{proof}
Let $S=\left\{ \tp_{L}\left(a / \emptyset \right):a\in\bsigma\left(M,\bar{b}\right)\right\} $.
For each $\sigma(x,b) \in \bsigma(x,b)$  we can find $\phi_{\sigma}(x) = \bigvee_{i<n} \phi_{p_i}(x)$  with $\phi_{p_i}(x) \in p_i$ for some $p_i \in S$, such that $\bsigma(x,b) \land \P(x) \rightarrow \phi_\sigma(x)$ and $\phi_{\sigma}(x) \land \P(x) \rightarrow \sigma(x,b) \land \P(x)$ holds (by compactness applied twice). Then $\bsigma(x,b) \cap \P = \bigcap_{\sigma(x,b) \in \bsigma(x,b)} \phi_\sigma (x) \cap \P(x)$.

\end{proof}
First we observe existence of $G^{00}$ relatively to $\P\left(x\right)$.
\begin{prop}
\label{thm: existence of relative G^00}Let $\left(N,M\right)$ be
a saturated pair.
\begin{enumerate}
\item For any small set $B\subset N$, we have $G_{L_{\P}\left(\emptyset\right)}^{00}\left(N,M\right)\subseteq G_{L\left(B\right)}^{00}\left(N,M\right)$.
\item In particular it follows that $G_{L}^{00}\left(N,M\right)=G_{L\left(B'\right)}^{00}\left(N,M\right)$
for some small $B'\subset N$, and $\left[M:G_{L}^{00}\left(N,M\right)\right]\leq2^{2^{\left|T\right|}}$.
\end{enumerate}
\end{prop}
\begin{proof}
All the formulas of the form $\phi\left(x,b\right)\land\P\left(x\right)$
with $\phi\left(x,y\right)\in L$ are $\NIP$. Then the usual proof
of the existence of $G^{00}$ in $\NIP$ theories, see e.g. \cite[Proposition 6.1]{NIP1},
goes through unchanged and gives that for a subgroup of $M=\P\left(N\right)$
of the form $\bsigma\left(M,\bar{b}\right)$ where $\bsigma\left(x,y\right)$
is a partial $L$-type, of bounded index, there are only boundedly
many different conjugates of it under $L_{\P}$-automorphisms. Then
the proposition follows by taking the intersection of all such conjugates,
over all $L_{\P}$-types of $\left|\bar{y}\right|$-tuples over $\emptyset$,
which is still of bounded index. 

To see the absolute bound on the index, note that $G_{L}^{00}\left(N,M\right)$
is invariant under $L_{\P}$-automorphisms and type-definable, so
it follows by saturation that it is $L_{\P}$-type-definable over
$\emptyset$. Then the bound follows by the usual application of Erd\H os-Rado.
\end{proof}
Now we will show that $G^{00}$ is not changed by adding externally
definable sets.
\begin{thm}
\label{thm: Collapse of G^00}Let $\left(N,M\right)$ be a saturated
pair. Then $G^{00}\left(M\right)=G_{L}^{00}\left(N,M\right)$.\end{thm}
\begin{proof}
Fix a cardinal $\lambda\gg2^{2^{\left|T\right|}}$. By induction on
$\alpha\leq\lambda$ we try to define $M_{\alpha}$, $N_{\alpha}$, $\bsigma_{\alpha}(x,\bar y_\alpha)$, $\bar{b}_{\alpha}$
such that:
\begin{enumerate}
\item \label{thm: Collapse of G^00, item1} $\left(M_{\alpha}\right)_{\alpha\leq\lambda}$ and $(N_{\alpha})_{\alpha \leq \lambda }$ are $L$-elementary chains
of models;

\item \label{thm: Collapse of G^00, item2} $\left(N_{\alpha},M_{\alpha}\right)$ is a saturated elementary pair;

\item \label{thm: Collapse of G^00, item3} $\bsigma_{\alpha}\left(x,\bar{y}_\alpha \right)$ is a partial $L$-type of
size bounded with respect to the saturation of $\left(N_{\alpha},M_{\alpha}\right)$;

\item \label{enu: type of parameters grows-1}$\bar{b}_{\alpha}\in N_{\alpha}$;
\item \label{enu: L_P type is constant.}For any $\alpha \leq \lambda$, $\tp_{L_{\P}^{\forall,\bdd}}\left( (\bar{b}_{i})_{i \leq \alpha} \right)$ is the same evaluated in any of the pairs $\left(N_{\beta},M_{\beta}\right)$ for
all $\alpha \leq \beta \leq \lambda$;
\item \label{enu: G_alpha,i is type-definable-1}For each $\alpha \leq \lambda$,  $\bsigma_\alpha\left(x,\bar{b}_{\alpha}\right)$
is a hereditary subgroup of $\P(x)$ of hereditarily bounded
index in $(N_{\alpha},M_{\alpha})$;
\item \label{thm: Collapse of G^00, item7} $\bsigma_{i}(M_\alpha, \bar b _i) \subsetneq \bsigma_{j} (M_\alpha, \bar b _j)$ for all $j<i\leq\alpha<\lambda$;
\item \label{thm: Collapse of G^00, item8} $M_0 = M, N_0 =N$ and $\bsigma_0(M_0,\bar{b}_0) = G^{00}_L(N,M)$.
\end{enumerate}
Suppose that we manage to carry out the induction. But then this means
that in a saturated pair $\left(N_{\lambda},M_{\lambda}\right)$ we
have a strictly decreasing sequence $\left(\bsigma(M_\lambda, \bar b _i ):i<\lambda\right)$
of subgroups of $M_{\lambda}$ of bounded index which are definable
by small $L$-types over small sets of parameters from $N_{\lambda}$
--- contradicting Proposition \ref{thm: existence of relative G^00}.

So let $\alpha^{*}$ be the smallest ordinal at which we got stuck.

~

\textbf{Claim 1}: $\alpha^{*}$ is a successor.

Proof: Assume that $\alpha^{*}=\bigcup_{\alpha<\alpha^{*}}\alpha$
is a limit ordinal. We then define $M'=\bigcup_{\alpha<\alpha^{*}}M_{\alpha}$,
and $N' = \bigcup_{\alpha<\alpha^{*}}N_{\alpha}$. Note that $(N',M')$ is an elementary pair. Let $(N_{\alpha ^*}, M_{\alpha ^*})$ be some saturated extension of $(N',M')$. Let $\bar{b}_{\alpha^*}=\bigcup_{\alpha < \alpha^*} \bar{b}_\alpha$ and $\bsigma_{\alpha^*}(x,\bar{b}_{\alpha^*}) = \bigcup_{\alpha < \alpha^*} \bsigma_\alpha(x,\bar{b}_\alpha)$.
Using Lemma \ref{lem: L-type over P determines hereditary properties in the pair}\ref{lem: L-type over P determines hereditary properties in the pair 1},\ref{lem: L-type over P determines hereditary properties in the pair 2} and the inductive assumption it is easy to verify that $(N_{\alpha ^*}, M_{\alpha ^*})$ and $\bsigma_{\alpha^*}(x,\bar{b}_{\alpha^*})$ satisfy all the requirements \ref{thm: Collapse of G^00, item1}--\ref{thm: Collapse of G^00, item8}. But this contradicts the choice of $\alpha^*$.

\if(0)

let $K\succ M'$ be very saturated. By saturation of $K$ and
the inductive assumption, for each $i<\alpha^{*}$ we can find $\bar{b}_{\alpha^{*},i}\models\bigcup_{i<\alpha<\alpha^{*}}\tp_{L}\left(\bar{b}_{\alpha,i}/M_{\alpha}\right)$.
By (\ref{enu: type of parameters grows-1}), (\ref{enu: L_P type is constant.}),
(\ref{enu: G_alpha,i is type-definable-1}) and Lemma \ref{lem: L-type over P determines hereditary properties in the pair}
it follows that $\tp_{L_{\P}^{\forall,\bdd}}\left(\bar{b}_{\alpha^{*},i}\right)=q_{i}\left(\bar{y}\right)$
and that $\bsigma_{i}\left(M',\bar{b}_{\alpha^{*},i}\right)$ is a
hereditary subgroup of $M'$ of hereditarily bounded index in the
pair $\left(K,M'\right)$. Let $\left(N_{\alpha^{*}},M_{\alpha^{*}}\right)$
be a saturated elementary extension of $\left(K,M'\right)$. Then
$G_{\alpha^{*},i}=\bsigma_{i}\left(M_{\alpha^{*}},\bar{b}_{\alpha^{*},i}\right)$
is still a hereditary subgroup of hereditarily bounded index, $\tp_{L_{\P}^{\forall,\bdd}}\left(\bar{b}_{\alpha^{*},i}\right)=q_{i}$
in the pair $\left(N_{\alpha^{*}},M_{\alpha^{*}}\right)$ as it is
an elementary extension of $\left(K,M'\right)$, and it is easy to
see using the inductive assumption that all the remaining conditions
are satisfied. But this contradicts the choice of $\alpha^{*}$, so
$ $the claim is proved.

\fi
~

So $\alpha^{*}=\alpha+1$ is a successor. Take $K \succ N_\alpha \succ M_{\alpha}$
very saturated. We let $\Lambda(x)$ be the union
of all hereditary subgroups of $M_{\alpha}$ of hereditarily bounded
index, in the sense of the pair $\left(K,M_{\alpha}\right)$, definable
by partial $L$-types with parameters from $K$ (recall that according to Definition \ref{def: hereditary subgroup}, a hereditary subgroup is a partial type, so we are taking a union of the types, which corresponds to taking the intersection of the groups defined by those types). Note that this union
might contain $2^{\left|M_{\alpha}\right|}$-many $L$-formulas,
but that's ok. Note that $\bar{b}_{\leq \alpha}$ has the same $L_{\P}^{\bdd}$-type in $(K,M_\alpha)$ as in $(N_\alpha, M_\alpha)$ (as it is determined by $\tp_L(\bar{b}_{\leq \alpha}/M_\alpha)$). In view of Lemma \ref{lem: L-type over P determines hereditary properties in the pair}\ref{lem: L-type over P determines hereditary properties in the pair 1} we have $\Lambda(M_\alpha) \subseteq G_L^{00}(N_\alpha, M_\alpha) \subseteq \bsigma_\alpha(M_\alpha, \bar{b}_\alpha)$.

\if(0)
By saturation of $K$ over $M_{\alpha}$, for every
$i<\alpha$ we find some $\bar{b}_{\alpha+1,i}$ in $K$ realizing
$\tp_{L}\left(\bar{b}_{\alpha,i}/M_{\alpha}\right)$, so in particular
$\tp_{L_{\P}^{\forall,\bdd}}\left(\bar{b}_{\alpha+1,i}\right)=\tp_{L_{\P}^{\forall,\bdd}}\left(\bar{b}_{\alpha,i}\right)=q_{i}$
in the corresponding pairs. By Lemma \ref{lem: L-type over P determines hereditary properties in the pair}\ref{lem: L-type over P determines hereditary properties in the pair 1}
it follows that $\bsigma_{i}\left(M_{\alpha},\bar{b}_{\alpha+1,i}\right)$
is a hereditary subgroup of $ $$M_{\alpha}$ of hereditarily bounded
index in the pair $\left(K,M_{\alpha}\right)$. Similarly we see that
$\Lambda\left(M_{\alpha}\right)\subseteq G_{L}^{00}\left(N_{\alpha},M_{\alpha}\right)$.
\fi
~

\textbf{Claim 2}: $\Lambda\left(M_{\alpha}\right)=G_{L}^{00}\left(N_{\alpha},M_{\alpha}\right)$.

Proof: Assume that $\Lambda\left(M_{\alpha}\right)\subsetneq G_{L}^{00}\left(N_{\alpha},M_{\alpha}\right)$.
Let $\left(N_{\alpha+1},M_{\alpha+1}\right)$ be a saturated
elementary extension of $\left(K,M_{\alpha}\right)$ and set $\bsigma_{\alpha+1}(x, \bar{b}_{\alpha+1}) = \Lambda(x)$ (which is a union of size $\leq2^{\left|M_{\alpha}\right|}$
of hereditary subgroups of hereditarily bounded index, thus is of
hereditarily bounded index and is relatively definable by an $L$-type
in $\left(N_{\alpha+1},M_{\alpha+1}\right)$).
It is now easy
to see using the inductive assumption and Lemma \ref{lem: L-type over P determines hereditary properties in the pair}\ref{lem: L-type over P determines hereditary properties in the pair 1} that all the conditions \ref{thm: Collapse of G^00, item1}--\ref{thm: Collapse of G^00, item8}
are satisfied for $\alpha^{*}=\alpha+1$, which means that we could
have continued the induction contradicting the choice of $\alpha$,
so the claim is proved.

\if(0)
and set $G_{\alpha+1,i}=\bsigma_{i}\left(M_{\alpha+1},\bar{b}_{\alpha+1,i}\right)$
for $i<\alpha$ (which is a hereditary subgroup of hereditarily bounded
index by the elementarity of the extension) and $G_{\alpha+1,\alpha}=\Lambda\left(M_{\alpha+1}\right)$
(which is defined an intersection of size $\leq2^{\left|M_{\alpha}\right|}$
of hereditary subgroups of hereditarily bounded index, thus is of
hereditarily bounded index and is relatively definable by an $L$-type
in $\left(N_{\alpha+1},M_{\alpha+1}\right)$). Note that we still
have $\tp_{L_{\P}^{\forall,\bdd}}\left(\bar{b}_{\alpha+1,i}\right)=q\left(\bar{y}\right)$
in the pair $\left(N_{\alpha+1},M_{\alpha+1}\right)$. $ $It is easy
to see using the inductive assumption that all the remaining conditions
are satisfied for $\alpha^{*}=\alpha+1$, which means that we could
have continued the induction contradicting the choice of $\alpha$,
so the claim is proved.
\fi
~

As every $L$-automorphism of $M_{\alpha}$ extends to an $L_{\P}$-automorphism
of the pair $\left(K,M_{\alpha}\right)$ by saturation of $K$, it
follows that $\Lambda\left(M_{\alpha}\right)$ is $\Aut_{L}\left(M_{\alpha}\right)$-invariant.
But by Claim 2 this means that $G_{L}^{00}\left(N_{\alpha},M_{\alpha}\right)$
is an $\Aut_{L}\left(M_{\alpha}\right)$-invariant subgroup of $M_{\alpha}$.
Along with Fact \ref{fac: type-def and invariant implies type-definable} and \ref{thm: Collapse of G^00, item7}
this implies that $G^{00}\left(M_{\alpha}\right)\subseteq G_{L}^{00}\left(N_{\alpha},M_{\alpha}\right)\subseteq \bsigma_{\alpha}\left(M_{\alpha},\bar{b}_{\alpha}\right) \subseteq \bsigma_{0}(M_\alpha,\bar{b}_0)$. As $G^{00}(M_\alpha)$ is $\bigwedge$-definable over $\emptyset$, we have $G^{00}(M_0) \subseteq \bsigma_0(M_0, b_0) = G^{00}_L (N_0,M_0)$ --- as wanted.

\end{proof}
\begin{cor}
\label{cor: G^00 is not changed in Shelah's expansion} Let $M\models T$, and
let $\widetilde{M}\succ M^{\ext}$ be a monster model for Shelah's
expansion of $M$, in the language $L'$. Then $G^{00}\left(\widetilde{M}\right)=G^{00}\left(\widetilde{M}\restriction L\right)$.\end{cor}
\begin{proof}
Let $N\succ M$ be $|M|^{+}$-saturated and, let $\left(N',M'\right)\succ \left(N,M\right)$ be $|N|^{+}$-saturated.
We may identify $\widetilde{M}$ with $M'$, in such a way that every
$\phi\left(\bar{x}\right)\in L'\left(\emptyset\right)$ is equivalent
on $M'$ to some $\psi\left(\bar{x}\right)\in L\left(N\right)$. Consider
$G^{00}\left(M'\right)$, as $\Th_{L'}\left(M'\right)$ is $\NIP$
by Shelah's theorem, it follows from existence of $G^{00}$ in $\NIP$
theories that $G^{00}\left(M'\right)$ is definable by a partial $L'$-type
over $\emptyset$, thus definable by a partial $L$-type over $N$.
That is, $G^{00}\left(\widetilde{M}\right)\supseteq G_{L}^{00}\left(N',M'\right)$,
and we can conclude by Theorem \ref{thm: Collapse of G^00}.\end{proof}
\begin{cor}
Let $\left(N,M\right)$ be an elementary pair, and assume that $H\left(M\right)$
is a (hereditary) subgroup of $M$ of hereditarily bounded index,
which is $L\left(N\right)$-type-definable. Then it is $L\left(M\right)$-type-definable.\end{cor}
\begin{proof}
Let $\left(N',M'\right)$ be a saturated extension of $\left(N,M\right)$.
By the theorem we know that $G^{00}_{L}\left(M'\right)\subseteq H\left(M'\right)$, and thus $H(M')$ is a union of a small set of cosets of $G^{00}(M')$, say $H(M') = \bigcup_{i<\lambda} \left( g_i \cdot G^{00}(M') \right)$ for some $g_i \in M'$ and $\lambda$ smaller than the saturation of $M'$. As $G^{00}_L$ is defined over $\emptyset$ and is of bounded index, we have that $a \equiv_M^{L} b \Rightarrow a \cdot G^{00}(M') = b \cdot G^{00}(M')$ for any $a,b\in M'$. But then, given $\sigma \in \Aut_L(M'/M)$, we see that $\sigma \left(H(M') \right) = \bigcup_{i<\lambda} \left( \sigma(g_i) \cdot G^{00}(M') \right) = \bigcup_{i<\lambda} \left( g_i \cdot G^{00}(M') \right) = H(M')$, i.e. $H(M')$ is $\Aut_L(M'/M)$-invariant. By Fact \ref{fac: type-def and invariant implies type-definable}
it follows that $H\left(M'\right)$ is $L\left(M\right)$-type-definable, and thus $H(M)$ as well.
\end{proof}

\subsection{$G^{\infty}$}

Again let $\left(N,M\right)$ be a saturated elementary pair of models of $T$ and let $L'$ be a collection of $L_{\P}$-formulas such that $L\subseteq L'\subseteq L_{\P}$.

\begin{defn}
We consider all subgroups of $\P\left(N\right)=M$ of bounded index
(that is of index less than the saturation) and definable as $\bsigma\left(M,B\right)$
where $B$ is a small tuple from $N$ and $\bsigma$ is a disjunction
of complete $L'$-types over $B$, each of which is consistent with $\P(x)$.
Let $G_{L'\left(B\right)}^{\infty}\left(N,M\right)$ be defined as
the intersection of all such groups, and let $G_{L'}^{\infty}\left(N,M\right)=\bigcap_{B\subset N,\mbox{small}}G_{L'\left(B\right)}^{\infty}\left(N,M\right)$.
\end{defn}
First we establish a version of the existence of $G^{\infty}$ relatively to a predicate
$\P\left(x\right)$.
\begin{thm}
\label{thm: relative G^infty exists}Let $\left(N,M\right)$ be a
saturated pair. Then:
\begin{enumerate}
\item For any small set $A\subset N$, we have $G_{L_{\P}\left(\emptyset\right)}^{\infty}\left(N,M\right)\subseteq G_{L_{\P}^{\bdd}\left(B\right)}^{\infty}\left(N,M\right)$.
\item In particular it follows that $G_{L}^{\infty}\left(N,M\right)=G_{L\left(B'\right)}^{\infty}\left(N,M\right)$
for some small $B'\subset N$, and $\left[M:G_{L}^{\infty}\left(N,M\right)\right]\leq2^{2^{\left|T\right|}}$.
\end{enumerate}
\end{thm}
\begin{proof}
Let $\bar{\kappa}$ be the
saturation of the pair (a strong limit, of large enough cofinality).

For a small set $A\subseteq N$, we define $X_{A}=\left\{ a^{-1}b:a,b\in\P\land a\equiv_{A}^{L_{\P}^{\bdd}}b\right\} $.
As usual, given sets $X,Y$, we denote $X^{Y}=\left\{ x^{y}:x\in X,y\in Y\right\} $
where $x^{y}=y^{-1}xy$, and $X^{n}=\left\{ x_{1}x_{2}\ldots x_{n}:x_{i}\in X\right\} $.

~

\textbf{Claim.} $\left(X_{A}\right)^{\P}\subseteq\left(X_A\right)^{2}$

Proof: Let $a\equiv_{A}^{L_{\P}^{\bdd}}b$ and $c$ from $\P$ be
arbitrary. By the assumption and compactness there is some $d\in\P$ such that
$\left(a,c\right)\equiv_{A}^{L_{\P}^{\bdd}}\left(b,d\right)$, so
in particular $c\equiv_{A}^{L_{\P}^{\bdd}}d$ and $ac\equiv_{A}^{L_{\P}^{\bdd}}bd$. Then we have $\left(X_A\right)^{\P}\ni\left(a^{-1}b\right)^{c}=\left(ac\right)^{-1}bc=\left(\left(ac\right)^{-1}\left(bd\right)\right)\left(d^{-1}c\right)\in\left(X_A\right)^{2}$.

~

Assume that $ G^{\infty}_{L_{\P}(\emptyset)}(N,M) \not \subseteq  G^{\infty}_{L_{\P}^{\bdd}(A)}(N,M)$ for some small $A \subseteq N$. Let $B$ satisfying $A \subseteq B \subseteq N, |B| \leq \lambda $ be a small set containing representatives of all cosets of all subgroups of $\P(x)$ of bounded index which are definable by disjunctions of complete $L_{\P}^{\bdd}$-types over $A$. 
Let $a \equiv_{B}^{L_{\P}^{\bdd}} b$ be arbitrary. By assumption there is some $c \in B$ from the same coset of $G^{\infty}_{L_{\P}^{\bdd}(A)}(N,M)$ as $a$. It follows that $b$ is from the same coset of $G^{\infty}_{L_{\P}^{\bdd}(A)}(N,M)$ as $c$, thus as $a$. But this implies that $a^{-1}b \in G^{\infty}_{L_{\P}^{\bdd}(A)}(N,M)$, and so 
$\langle X_B \rangle \subseteq G^{\infty}_{L_{\P}^{\bdd}(A)}(N,M)$. Let $X = \bigcap_{B' \subset N, |B'| \leq \lambda} \langle X_{B'} \rangle$. Note that $X$ is invariant with respect to $L_{\P}$-automorphisms of the pair (over $\emptyset$). So, if $X$ had bounded index in $\P$, we would have $G^{\infty}_{L_{\P}(\emptyset)}(N,M) \subseteq X \subseteq \langle X_B \rangle $, a contradiction.
 
Thus $X$ has unbounded index in $\P$, and we can find arbitrary long sequences $\left( B_i, c_i \right)_{i \in \kappa}$ with $B_i \subset N, |B_i|\leq \lambda$ and $c_i \in \P$ such that $c_i \in \left( \bigcap_{j<i} \langle X_{B_j} \rangle \right) \setminus \langle X_{B_i} \rangle$.
 By Erd\H os-Rado we may find such a sequence which is moreover $L_{\P}$-indiscernible. In particular, for some $m \in \omega$ we have:
\begin{enumerate}
 \item [{(i)}]$c_i \in \left( \bigcap_{j<i} X_{B_j}^m \right) \setminus X_{B_i}^{m+4}$ for all $i$.
\end{enumerate}

Next, using $L_{\P}$-indiscernibility of the sequence, the claim and
compactness we can find some finite sets of formulas $\Phi,\Phi'\subset L_{\P}^{\bdd}$
such that:
\begin{enumerate}
\item [{(ii)}] $c_{i}\notin\left(X_{B_i,\Phi}\right)^{m+4}$,
\item [{(iii)}] $\left(X_{B_{i},\Phi'}\right)^{\P}\subseteq\left(X_{B_{i}, \Phi}\right)^{2}$.
\end{enumerate}
(where $X_{A, \Phi }=\left\{ a^{-1}b:\bigwedge_{\phi\in\Phi}\left(\phi\left(a,A\right)\leftrightarrow\phi\left(b,A\right)\right), a,b \in \P \right\} $).

Now, for an arbitrary increasing finite sequence of natural numbers $I=\left(i_{1},\ldots,i_{n}\right)$ 
we define the following elements
of $\P$:
\begin{itemize}
\item $c_{I,0}=c_{2i_{1}+1}\cdot\dots\cdot c_{2i_{n}+1}$,
\item $c_{I,1}=c_{2i_{1}}\cdot\ldots\cdot c_{2i_{n}}$.
\end{itemize}
To obtain a contradiction it is sufficient to show:
\begin{enumerate}
\item [{(iv)}] if $j\notin I$, then $c_{I,0}c_{I,1}^{-1}\in X_{B_{2j}}\subseteq X_{B_{2j},\Phi'}$,
\item [{(v)}] if $j\in I$, then $c_{I,0}c_{I,1}^{-1}\notin X_{B_{2j},\Phi'}$.
\end{enumerate}
(as then the bounded formula $\psi(x,\bar{y}) := \exists z_1, z_2 \in \P \left( x = z_1^{-1} z_2 \land \bigwedge_{\phi \in \Phi'} \left( \phi(z_1,\bar{y}) \leftrightarrow \phi(z_2,\bar{y}) \right) \right)$, $\psi(x,\bar{y}) \in L_{\P}^{\bdd}$ would have $\IP$ over $\P$ in $\left(N,M\right)$, witnessed by
$\left(\left(c_{I,0}c_{I,1}^{-1}\right):I\subset\omega\right)$
in $\P$ and $\left(B_{2j}:j<\omega\right)$ from $N$, contradicting
Remark \ref{rem: bounded formulas are NIP}).

So if $j\notin I$, then $c_{I,0}\equiv_{B_{2j}}^{L_{\P}}c_{I,1}$
by $L_{\P}$-indiscernibility of our sequence,
thus $c_{I,0}c_{I,1}^{-1}\in X_{B_{2j}}$ and (iv) follows.

Assume that (v) does not hold, then one gets a contradiction exactly like in \cite[Theorem 5.3]{Jakub}. Indeed, then for some $j \in I$ we have $c_{I,0}c_{I,1}^{-1} \in X_{B_{2j}, \Phi'}$. Let $I=I_1^{\smallfrown} \{j\}^{\smallfrown} I_2$, then:
$$ c_{I,0} \cdot c_{I,1}^{-1} = c_{I_1,0} \cdot c_{2j+1} \cdot c_{I_2,0} \cdot c_{I_2,1}^{-1} \cdot c_{2j}^{-1} \cdot c_{I_1,1}^{-1} ,$$
$$ c_{I,0} \cdot c_{I,1}^{-1} \cdot c_{I_1,1} \cdot c_{2j} = c_{I_1,0} \cdot c_{2j+1} \cdot c_{I_2,0} \cdot c_{I_2,1}^{-1},$$
$$ c_{2j} = c_{I_1,1}^{-1} \cdot c_{I,1} \cdot c_{I,0}^{-1} \cdot c_{I_1,0} \cdot c_{2j + 1} \cdot c_{I_2,0} \cdot c_{I_2,1}^{-1} = $$
$$= \left( c_{I_1,1}^{-1} \cdot \left( c_{I,1} \cdot c_{I,0}^{-1} \right) \cdot c_{I_1,1} \right) \cdot \left( c_{I_1,1}^{-1} \cdot c_{I_1,0} \right) \cdot c_{2j+1} \cdot \left( c_{I_2,0} \cdot c_{I_2,1}^{-1} \right).$$
Since $j \notin I_1 \cup I_2$, by (iv) we have $\left( c_{I_1,1}^{-1} \cdot c_{I_1,0} \right), \left( c_{I_2,0} \cdot c_{I_2,1}^{-1} \right) \in X_{B_{2j}}$. By the assumption $c_{I,0} \cdot c_{I,1}^{-1} \in X_{B_{2j},\Phi'}$ and $c_{2j+1} \in X_{B_{2j}}^m$. Combining and using (iii) we obtain
$$ c_{2j} \in X_{B_{2j}, \Phi'}^{c_{I_1,1}} \cdot X_{B_{2j}}^{m+2} \subseteq X_{B_{2j}, \Phi}^{m+4},$$

contradicting (ii).

\if(0)
Now we procceed to the proof of the theorem. Let $B\subset N$ be
an arbitrary small set. For every subgroup of $\P$ of bounded index
which is definable as a disjunction of $L$-types over $B$, the set
of representatives of all of its cosets in $\P$ is small. As there
are also only boundedly many subgroups of $\P$ of this form, there
is a small set \textbf{$B'\subset N$ }containing $B$ and representatives
of all cosets for each of such groups. Let $\alpha=\left|B'\right|^{+}$.
By Claim 2, for every $m<\omega$ there is some $\lambda_{m}<\bar{\kappa}$
and some family $\bar{A}=\left\{ A_{i}:i<\lambda_{m}\right\} $. Since
$\bar{\kappa}$ has uncountable cofinality, we can actually find some
$\lambda<\bar{\kappa}$ and $\left\{ A_{i}:i<\lambda\right\} $ which
works for all $m<\omega$, and moreover is closed under finite unions
and $B'\in\bar{A}$. Define $S_{m}=\bigcap_{i<\lambda}X_{\equiv_{A_{i}}}^{m}$
and $T_{m}=\bigcap_{A\subset N,\left|A\right|<\alpha}X_{\equiv_{A}}^{m}$.
Then we have:
\begin{enumerate}
\item $T_{m}\subseteq S_{m}\subseteq T_{m+4}\subseteq\P$ (by Claim 2)
\item $S_{m}=\left(S_{1}\right)^{m}$ (by compactness in $\left(N,M\right)$,
as $\left\{ A_{i}\right\} $ is closed under finite unions)
\item $S_{m}\subseteq S_{n}$ for $n\leq n$ (as $e=e^{-1}e\in S_{1}$)
\item $S_{1}\subseteq X_{\equiv_{B'}}$ (as $B'\in\bar{A}$)
\item $X_{\equiv_{\bigcup_{i<\lambda}A_{i}}}\subseteq S_{1}$ (clear)
\end{enumerate}
Let $H=\bigcup_{m<\omega}S_{m}=\bigcup_{m<\omega}T_{m}$ (by (1)).
And so finally we have:
\begin{itemize}
\item $H$ is a subgroup of $\P$ (given $a\in S_{n}$ and $b\in S_{m}$
we have that \textbf{$ab\in S_{n+m}$ }and $a^{-1}\in S_{n}$)
\item $H$ is $\Aut_{L_{\P}}\left(\emptyset\right)$-invariant (as $T_{m}$
is $\Aut_{L_{\P}}$-invariant for every $m<\omega$)
\item $H$ is of bounded index in $\P$ (consider the group $G=\left\langle X_{\equiv_{\bigcup_{i<\lambda}A_{i}}}\right\rangle $
generated by $X_{\equiv_{\bigcup_{i<\lambda}A_{i}}}$, by (5) it is
a subgroup of $H$. $G$ has only boundedly many cosets in $\P$ as
for any $a,b$ from $\P$, if $a^{-1}b\notin G$ then $a$ and $b$
have different $L_{\P}$-types over $\bigcup_{i<\lambda}A_{i}$, of
which there are boundedly many --- so $H$ is of bounded index as
well).
\item $H\subseteq G_{L\left(B\right)}^{\infty}\left(N,M\right)$ (Let $G$
be some subgroup of $\P$ of bounded index, which is relatively definable
by a disjunction of $L$-types over $B$. As $H=\left\langle S_{1}\right\rangle $
by (2), it is enough to show that $S_{1}\subseteq G$. So let $a\equiv_{B'}^{L_{\P}^{\bdd}}b$
be given. By the choice of $B'$ there is some $c\in B'$ from the
same coset of $G$ as $a$. As $Bc\subseteq B'$, it follows that
$b$ is from the same coset of $G$ as $c$, thus from the same coset
as $a$. But this implies that $a^{-1}b\in G$ --- as wanted).
\end{itemize}
Thus the theorem is proved.
\fi

\end{proof}
\begin{prop}
Let $\left(N,M\right)$ be a saturated pair. Then $G_{L}^{\infty}\left(N,M\right)=G^{\infty}\left(M\right)$.\end{prop}
\begin{proof}
We can repeat the proof of Theorem \ref{thm: Collapse of G^00} using
Theorem \ref{thm: relative G^infty exists} instead of Theorem \ref{thm: existence of relative G^00},
but here is a shorter argument.

Let $K\succ M$ be $|M|^{+}$-saturated, and let $\left(K',M'\right)\succ \left(K,M\right)$ be $|K|^{+}$-saturated.
Let $\Lambda\left(M\right)$ be the intersection of all hereditary
subgroups of $M$ of hereditarily bounded index (in the sense of the
pair $\left(K,M\right)$) definable by disjunctions of $L$-types
with parameters from $K$. By saturation of $K$ over $M$ and Lemma
\ref{lem: L-type over P determines hereditary properties in the pair}\ref{lem: L-type over P determines hereditary properties in the pair 1}
it follows that $\Lambda\left(M\right)\subseteq G_{L}^{\infty}\left(N,M\right)$.

Note that $G_{L}^{\infty}\left(K',M'\right)\subseteq\Lambda\left(M'\right)$,
so by Theorem \ref{thm: relative G^infty exists} we have $\left[M':\Lambda\left(M'\right)\right]\leq2^{2^{\left|T\right|}}$,
which implies $\left[M:\Lambda\left(M\right)\right]\leq2^{2^{\left|T\right|}}$.
Now $\Lambda\left(M\right)$ is an $L$-invariant subgroup of $M$
(by saturation of $K$ every $L$-automorphism of $M$ extends to
an $L_{\P}$-automorphism of $\left(K,M\right)$) and of index smaller
than the $L$-saturation of $M$, so $G^{\infty}\left(M\right)\subseteq\Lambda\left(M\right)\subseteq G_{L}^{\infty}\left(N,M\right)$.\end{proof}
\begin{cor}
\label{cor: G000(M)=G000(Mext)}
Let $M\models T$,
and let $\widetilde{M}\succ M^{\ext}$ be a monster model for Shelah's
expansion of $M$, in the language $L'$. Then $G^{\infty}\left(\widetilde{M}\right)=G^{\infty}\left(\widetilde{M}\restriction L\right)$.\end{cor}
\begin{proof}
Same as the proof of Corollary \ref{cor: G^00 is not changed in Shelah's expansion}.
\end{proof}

\begin{problem}
\begin{enumerate}
\item
By \cite[Remark 8.3]{NIP2} if $G$ is a definably amenable $\NIP$ group such that $G/G^{00}$ is a compact Lie group, then $G^{00}$ is externally definable. Is there any generalization of this fact for arbitrary $\NIP$ groups, or at least for the finite dp-rank case? E.g., does naming $G^{00}$ by a predicate preserve
$\NIP$?
\item
In view of the results of this section, one can try to understand various connected components and quotients
in an elementary pair of models in terms
of the base theory.
\end{enumerate}
\end{problem}

\section{Topological dynamics and the ``Ellis group'' conjecture}\label{sec: TopDyn and Ellis}

\subsection{Topological dynamics and minimal flows} \label{sec: general top dyn}

The subject topological dynamics tries to understand a topological group via its actions
on compact spaces. A good reference is \cite{Auslander}. As originally suggested by Newelski, topological dynamics yields new insights into the model theory of definable
groups, as well as new invariants, which are especially relevant to generalizing stable group theory to other ``tame'' contexts, such as 
groups in $\NIP$ theories. In \cite{AnandTopDyn,GisPenPil}, a theory of ``definable'' topological dynamics was developed, following earlier work of Newelski.

The context is: a model $M_{0}$ and a group $G$ (identified with its points in a saturated elementary extension of $M_{0}$) which is definable over $M_{0}$. 

ASSUMPTION: All types in $S_{G}(M_{0})$ are definable. 

Two extreme cases are:

(a) $M_{0}$ is the standard model of set theory, and $G(M_{0})$ is a group, 

(b) $T$ is an $\NIP$ theory, $M\models T$, $G$ a group definable over $M$, and $M_{0} = M^{\ext}$. 

In case (a) our theory reduces to the classical topological dynamics of the discrete group $G(M_{0})$. In case (b) which is the interest of the current paper, we at least obtain some new invariants and problems. We summarize the theory developed in \cite{GisPenPil}, as background for the results of this section.

We call a map $f$ from $G(M_{0})$ to a compact space $C$ \emph{definable} if for any disjoint closed sets $C_{1}$, $C_{2}$ of $C$, $f^{-1}(C_{1})$ and 
$f^{-1}(C_{2})$ are separated by a definable set. An action of $G(M_{0})$ on a compact space $C$ (by homeomorphisms) is ``definable'' if for any $x\in C$, the map from $G(M_{0})$ to $C$ which takes $g\in G$ to $gx$ is definable. Such actions are called definable $G(M_{0})$-flows.

\begin{fact}
\label{fac: topological dynamics} 
\begin{enumerate}
\item The left action of $G$ on $S_{G}(M_{0})$ is definable. Moreover $(S_{G}(M_{0}),1)$ is the (unique) universal definable $G(M_{0})$-ambit; where by a definable $G(M_{0})$-ambit we mean a definable $G(M_{0})$-flow $X$ with a distinguished point $x$ whose orbit is dense.
\item $S_{G}(M_{0})$ has a semigroup structure $\cdot$, which extends the group operation on $G(M_{0})$ and is continuous in the first coordinate. For $p,q\in S_{G}(M_{0})$, $p\cdot q$ is $\tp(a\cdot b/M_{0})$ where $b$ realizes $q$ and $a$ realizes the unique coheir of $p$ over $M_{0},b$.
\item Left ideals of $S_{G}(M_{0})$ are precisely closed $G(M_0)$-invariant subspaces (i.e. subflows of the definable $G(M_{0})$-flow $S_{G}(M_{0})$).
\item There is a unique (up to isomorphism) minimal definable $G(M_{0})$-flow ${\mathcal M}$, which coincides with some/any minimal subflow of $S_{G}(M_{0})$.
\item  Pick a minimal subflow ${\mathcal M}$ of $S_{G}(M_{0})$ and an idempotent $u\in {\mathcal M}$. Then $u\cdot{\mathcal M}$ is a subgroup of the semigroup $S_{G}(M_{0})$, whose isomorphism type does not depend on the choice of ${\mathcal M}$ or $u$. We call $u\cdot{\mathcal M}$ the Ellis group attached to the data. It also has a certain compact $T_{1}$ topology, with respect to which the group structure is separately continuous, but this will not really concern us here.
\item Using these ideas, the notions of definable amenability and definable extreme amenability can be characterized in a fashion similar to their characterization in the discrete case (e.g. a definable group $G$ is definably extremely amenable if and only if every definable action of it has a fixed point).
\end{enumerate}
\end{fact}

\subsection{Almost periodic types}\label{AP and WGen types}

\begin{defn}
A type $p\in S_{G}(M_{0})$ is called {\em almost periodic} if the closure of the orbit (under $G(M_{0})$) of $p$ is a minimal $G(M_{0})$-flow. 
\end{defn}

The usual characterization of almost periodicity holds:

\begin{fact} \label{fac: almost periodicity} 
The following are equivalent for a type $p\in S(M_0)$:
\begin{enumerate}
\item \label{fac: almost periodicity 1} $p(x)$ is almost periodic.
\item \label{fac: almost periodicity 2} For every $\phi(x) \in p$, the set $\overline{Gp}$ is covered by finitely many left translates of $\phi(x)$.
\item \label{fac: almost periodicity 3} For every formula $\phi(x)\in p$, $\{g\in G(M_{0}):g\phi \in p\}$ (which is a definable subset of $G(M_{0})$ by definability of $p$) is right generic, namely finitely many right translates cover $G(M_{0})$. 
\end{enumerate}
\end{fact}
\begin{proof} 
The equivalence of \ref{fac: almost periodicity 1}  and \ref{fac: almost periodicity 2}  holds by e.g. \cite[Remark 1.6]{New4}.

\ref{fac: almost periodicity 3}  $\Rightarrow$ \ref{fac: almost periodicity 1}: Suppose \ref{fac: almost periodicity 3}  holds. Suppose $q\in \overline{G(M_{0})p}$. Let $\phi(x)\in p$. Let 
$Z = \{g\in G(M_{0}):g\phi\in p \}$, so $Zg_1 \cup .. \cup Z g_n = G(M_{0})$ for some $g_{1},..,g_{n}\in G(M_{0})$.
Hence $g_{1}^{-1}\phi \vee ... \vee g_{n}^{-1}\phi  \in gp$ for all $g\in G(M_{0})$, whereby some $g_{i}^{-1}\phi\in q$, so $\phi \in g_{i}q$. We have shown that $p\in \overline{G(M_{0})q}$. So $\overline{G(M_{0})p}$ is minimal.
\newline
\ref{fac: almost periodicity 1} $\Rightarrow$ \ref{fac: almost periodicity 3}: Suppose $p$ is almost periodic, $\phi\in p$ and $Z =\{g\in G(M_{0}):g\phi\in p\}$. Let $\mathcal M = \overline{G(M_{0})p}$. Now by \ref{fac: almost periodicity 2} there are $g_{1},..,g_{n}\in G(M_{0})$ such that the clopen set $g_{1}\phi \vee ..\vee g_{n}\phi$ includes $\mathcal M$. It follows from the definition of $Z$ that $Z g_1^{-1} \cup ...\cup Z  g_n^{-1} = G_{0}(M)$. 

\end{proof}

\subsection{The Ellis group conjecture}
\medskip
\noindent
The quotient map from $G$ to $G/G^{00}_{M_{0}}$ factors through the tautological map $g\to \tp(g/M_{0})$ from $G$ to $S_{G}(M_{0})$,
and we let $\pi$ denote the resulting map from $S_{G}(M_{0})$ to $G/G^{00}_{M_{0}}$. It was pointed out in \cite{GisPenPil} that $G/G^{00}_{M_{0}}$ is the ``universal definable compactification'' of $G(M_0)$, which in case (a) from Section \ref{sec: general top dyn} is what is called the Bohr compactification of the discrete group $G(M_{0})$.

Let us note:

\begin{rem} The map $\pi$ is a surjective semigroup homomorphism, and for any minimal subflow ${\mathcal M}$ of $S_{G}(M_{0})$ and idempotent $u\in {\mathcal M}$, the restriction of $\pi$ to $u\cdot{\mathcal M}$ is surjective, hence a surjective group homomorphism.
\end{rem}
\begin{proof} The only thing possibly requiring a proof is the surjectivity of the restriction of $\pi$ to $u\cdot {\mathcal M}$. Let $g\in G$ and $p = \tp(g/M_{0})$. Then $p\cdot u \in {\mathcal M}$ as the latter is a left ideal. Hence $u\cdot (p \cdot u) \in {\mathcal M}$ and as $\pi(u)$ is the identity of $G/G^{00}_{M_{0}}$ we see that $\pi(u\cdot p \cdot u) = \pi(p)$.

\end{proof}

We now restrict to case (b) above: namely $T$ is $\NIP$, $G$ is a group definable over a model $M\models T$ and $M_{0} = M^{\ext}$. We make free use of the results from the previous sections, namely the preservation of various properties and objects associated to $G$ (definable amenability, $G/G^{00}$, etc) when passing from $T$ to $Th(M_{0})$. 

\medskip
\noindent
{\em Ellis group conjecture.} Suppose $G$ is definably amenable. Then the restriction of $\pi: S_{G}(M_{0}) \to G/G^{00}$ to $u\cdot{\mathcal M}$ is an isomorphism (for some/any choice of minimal subflow ${\mathcal M}$ of $S_{G}(M_{0})$ and idempotent $u\in {\mathcal M}$). 

\medskip
We remark that without the definable amenability assumption this statement is not true even for groups definable in $o$-minimal theories, see \cite{GisPenPilSL2R}. In the following subsections, we will prove (or explain) the announced cases of the conjecture, thus establishing Theorem \ref{thm: Ellis group conjectures}.
\noindent

\subsection{$G$ is definably extremely amenable, proof of Theorem \ref{thm: Ellis group conjectures}\ref{thm: EllisDefExtrAm}}
\label{sec: EllisStarts}

It was observed in Proposition \ref{prop: def ext am iff} that if $G$ is definably extremely amenable then $G=G^{00}$. On the other hand by definition of extreme amenability $u\cdot\mathcal{M}=\left\{ u\right\} $ for any minimal subflow $\mathcal{M}$ and idempotent $u\in\mathcal{M}$.

\subsection{$G$ is fsg, Theorem \ref{thm: Ellis group conjectures}\ref{thm: EllisFSG} } \label{sec: Ellis for fsg}
This was essentially proved in \cite{AnandTopDyn} (Theorem 3.8). Recall that  $G$ is $\fsg$ if it has a global $fsg$ type, namely a global type every translate of which is finitely satisfiable in $M$. We summarize the situation for the sake of completeness:
\begin{fact}
Let $G$ be $\fsg$. Then we have:
\begin{enumerate}
\item A global type is left (equivalently, right) generic if and only if it is left (equivalently, right) $f$-generic over $M_0$, if and only if every translate is finitely satisfiable in $M_0$.
\item There is a unique minimal subflow of $S_G(M_0)$, namely the set of generic types, and the Ellis group conjecture holds.
\end{enumerate}
\end{fact}

On the face of it the proof of Theorem 3.8 in \cite{AnandTopDyn}  depended on ``generic compact domination'' for $fsg$ groups from 
\cite{HruPilSimMeas}, but this was only 
required to deduce that $G$ is $fsg$ in $Th(M_{0})$ which we have already established in Theorem 3.19 of the current paper by direct means.  
So generic compact domination for $fsg$ groups, the proof of which in \cite{HruPilSimMeas} is incomplete, is not needed. 

\subsection{$G$ admits a definable $f$-generic, proof of Theorem \ref{thm: Ellis group conjectures}\ref{thm: EllisDefFGen}}
The other ``extreme case'' of definable amenability is when there is a global $f$-generic type, definable over $M$. We expect that if $G$ has some global definable $f$-generic type, then there is one which is definable over $M$. This feature was also considered by Hrushovski in \cite{HrushMetastable}, under the name ``groups with definable generics'' and Example 6.30 of that paper gives several examples from the theory of algebraically closed valued fields.

\begin{prop} \label{prop: definable f-generics}   Suppose $G$ has a global $f$-generic type, definable over $M_0$. Then
\newline
(i) $G^{00} = G^{0}$.
\newline
(ii) For $p\in S_{G}(M_{0})$, $p$ is almost periodic if and only if the global heir of $p$ is $f$-generic.
\newline
(iii) Any minimal subflow $\mathcal M$ of $S_{G}(M_{0})$ is already a group, so coincides with the Ellis group.
\newline
(iv) The Ellis group conjecture holds:  the restriction of $\pi$ to ${\mathcal M}$ is an isomorphism with $G/G^{0}$. 
\end{prop}
\begin{proof} (i)   Working in $T$, let  $q$ be a global $f$-generic definable over $M$ (or just definable). By Fact \ref{fac: properties of def am groups}, the left stabilizer of $q$ is $G^{00}$. But this left stabilizer is clearly an intersection of $M$-definable subgroups: for each $\phi(x,y)\in L$, $\Stab_{\phi}(q) = \{g\in G: \phi(x,c)\in q$ iff $\phi(g^{-1} x,c)\in q$ for all $c\}$. So each $\Stab_{\phi}(p)$ is finite index, whereby $G^{00} = G^{0}$. 

\vspace{2mm}
\noindent
(ii)  First assume that $p\in S_{G}(M_{0})$ and that the global heir $\bar p$ of $p$ is $f$-generic. We will use Fact \ref{fac: almost periodicity}. Let $\phi(x)\in p$. Then $X = \{g\in G: g\phi\in \bar p\}$ is definable over $M_{0}$ (by definability over $M_{0}$ of $\bar p$). Now $X$ contains the left stabilizer of $\bar{p}$ which, by Fact \ref{fac: properties of def am groups}, is $G^{00}$. As $G^{00}$ has bounded index in $G$ (and is a normal subgroup) and $X$ is definable, finitely many \emph{right} translates of $X$ cover $G$. Hence as $X$ is definable over $M_{0}$ the same thing is true in $G(M_{0})$.

The converse is a little more complicated. First by \ref{thm: fsg and definable generics}, there is a global $f$-generic ${\bar p}$ of $G$ with respect to $Th(M_{0})$ which is definable over $M_{0}$. Let $p$ be the restriction of ${\bar p}$ to $M_{0}$, so ${\bar p}$ is the unique global heir of $p$ and by the first part of the proof, $p$ is almost periodic. Let $I = \overline{G(M_{0})p}$. We first note that for any $q\in I$ the global heir of $q$ is $f$-generic. This is because $q = \tp(ab/M_{0})$ where $a\in G$ and $b$ realizes the unique heir of $p$ over $M_{0},a$ by Fact \ref{fac: topological dynamics}. But then the unique global heir of $q$ is precisely $a{\bar p}$ which we know to be $f$-generic and definable. 

Now let $q\in S_{G}(M_{0})$ be an almost periodic type, not necessarily in $I$. Let $J = \overline{G(M_{0})q}$.  By what we 
saw in the last paragraph, it suffices to show that some $r\in J$ has the required property. From material in Section 3 of \cite{GisPenPil}, the map from $I$ to $J$ which takes $p'\in I$ to $p'\cdot q$ is an isomorphism of 
$G(M_{0})$-flows, namely a homeomorphism which commutes with the action of $G(M_{0})$.  Let $r = p\cdot q$, and we show that $r$ (or its global heir) is as required, which will be enough.  We let $L$ denote the language of the structure $M_{0}$.
\newline
{\em Claim.} For any $L$-formula $\phi(x,y)$, $\Stab_{\phi}(r) = \{g\in G(M_{0}):$ for all $c\in M_{0}$, $\phi(x,c)\in r \leftrightarrow g\phi(x,c)\in r\}$ is a definable subgroup of $G(M_{0})$ of finite index. 

\vspace{2mm}
\noindent

Granted the claim, let $\bar r$ be the unique global heir of $r$ (i.e. defined by the same defining schema), and 
we see that $\Stab({\bar r}) = G^{0}$. Hence $\overline{r}$ is a global $f$-generic type definable over $M_{0}$, and we are 
finished.

\vspace{2mm}
\noindent
{\em Proof of Claim.}
Definability is immediate, by definability of the type $r$. 
Let $g_{1}$ realize $p$, and $a$ realize the unique heir of $q$ over $M_{0},g_{1}$, So $g_{1}a$ realizes $r$. 
Let $g\in G(M_{0})$, $c\in M_{0}$ and $\phi(x,y)\in L$. 
Let $\psi(z,y)$ be the $\phi(zx,y)$-definition for $q$. 
Then $g\phi(x,c)\in r$ iff $\phi(g^{-1}x,c)\in r$ iff $\models\phi(g^{-1}g_{1}a,c)$ iff $\phi((g^{-1}g_{1})x,c)$ is in 
the unique heir of $q$ over $M_{0},g_{1}$ iff $\models \psi(g^{-1}g_{1},c)$ iff $g\psi(x,c)\in p$. 

Hence $\Stab_{\phi}(r)  = \Stab_{\psi}(p)$. As $\bar p$ is $f$-generic and definable over $M_{0}$, its stabilizer has bounded index, hence  $\Stab_{\psi}(\bar p)$ which is an $M_{0}$-definable subgroup, has finite index.  Completing the proof of (ii). 

\vspace{2mm}

For the rest, we prove (iii) and (iv) simultaneously. We have seen have that any minimal (closed $G(M_{0})$-invariant) 
subflow $\mathcal M$ of $S_{G}(M_{0})$ has the form $\overline{G(M_{0})p}$ for $p\in S_{G}(M_{0})$ such that the global heir 
${\bar p} \in S(M')$ of $p$ has stabilizer equal to $G^{0}(M')$.  Fix such ${\mathcal M}$ and $p$. Let $a$ 
realize $\bar p$ (in a bigger model). Then ${\mathcal M} = \{\tp(ga/M_{0}):g\in G(M')\}$. As $\Stab({\bar p}) = 
G^{0}(M')$ it is easy to see that the elements of ${\mathcal M}$ are in natural $1-1$ correspondence with the cosets of 
$G^{0}$ in $G$. This suffices.
\end{proof}

\subsection{$G$ is dp-minimal, proof of Theorem \ref{thm: Ellis group conjectures}\ref{thm: EllisDPMin}}

Recall that a (partial) type $p$ over a set $A$ is dp-minimal if for any $a \models p$ and sequences $I_0,I_1$ mutually indiscernible over $A$, there is $i \in \{ 0,1 \}$ such that $I_i$ is indiscernible over $aA$ (see e.g. \cite{SimonDPmin}). We say that a definable group is dp-minimal if it is such as a definable set. It is easy to see that every extension of a dp-minimal type is dp-minimal.
\begin{prop}
\label{prop: dp-min def amenable group dichotomy} Let $G$ be a definably amenable, dp-minimal
group definable over $M_0$. Then either $G$ has $\fsg$ (witnessed as usual over $M_0$), or it has
a definable global $f$-generic type, definable over $M_0$.
\end{prop}
\begin{proof}
Firstly, the existence of a global $G$-invariant Keisler measure yields trivially
a $G(M_0)$-invariant Keisler measure $\mu$ over $M_0$ (i.e. on $M_0$-definable subsets of $G$). By Proposition \ref{prop: extending to a G-inv M-inv measure}, $\mu$ extends to a global $G$-invariant Keisler measure $\mu'$ which is definable over $M_0$. Let $p'$ be a global type in the support of $\mu'$. So $p'$ is $\Aut\left(\M/M_0\right)$-invariant. 
It is proved in \cite{SimonDPmin} that a dp-minimal global type invariant
over $M_0$ is either definable over $M_0$, or finitely satisfiable in
$M_0$.
Now any global type $p'$ in the support of $\mu'$ is $f$-generic and $M_0$-invariant. If some such type $p'$ is definable
over $M_0$ then we have our global $f$-generic, definable over $M_0$. Otherwise all global $p'$ in the support of $\mu$ are finitely satisfiable in $M_0$, whereby $G$ is $\fsg$ (with respect to $M_0$). 
\end{proof}

So we can derive part \ref{thm: EllisDPMin} of Theorem \ref{thm: Ellis group conjectures}:
\begin{cor}
Let $G$ be a definably amenable, dp-minimal group and $M$ any model over which $G$ is defined. Then $G/G^{00}$
coincides with the Ellis group  (computed over $M^{\ext}$).\end{cor}
\begin{proof}
It is obvious from the definition of dp-minimality and the fact that $Th(M^{\ext})$ has quantifier elimination, that  $G$ remains dp-minimal  in $Th(M^{\ext})$, so we can apply Proposition \ref{prop: dp-min def amenable group dichotomy} together with the parts \ref{thm: EllisFSG} and \ref{thm: EllisDefFGen} of Theorem \ref{thm: Ellis group conjectures} which have already been proved.
\end{proof}

\begin{problem}
Is it true that every dp-minimal group is definably amenable? More specifically, is it true that every dp-minimal group is nilpotent-by-finite?
\end{problem}

\begin{rem}
We would expect the Corollary to be true of definably amenable groups of finite dp-rank, by for example finding a composition series of $G$ where the factors are $\fsg$ or have definable $f$-generics. We will see in our proof of part (4) of Theorem \ref{thm: Ellis group conjectures} that this strategy works in the $o$-minimal case.

On the other hand it is not the case that a definably amenable group of  finite dp-rank contains a definable subgroup (or definable quotient) which is dp-minimal. For example, take a real closed field $R$ and take $G$ to be the group of $G$ points of a simple abelian variety over $R$ of algebraic-geometric dimension $>1$. Then $G$ has finite dp-rank, has $o$-minimal dimension $> 1$ so could not be dp-minimal, and has no proper definable subgroups. 
\end{rem}

\subsection{The $o$-minimal case, proof of Theorem \ref{thm: Ellis group conjectures}\ref{thm: EllisOMin}}
\label{EllisEnds}
The remainder of the paper is devoted to proving part (4) of \ref{thm: Ellis group conjectures}, the $o$-minimal case.
So we let $T$ be an $o$-minimal expansion of a real closed field, $M\models T$ and $G$ a (definably connected) definably amenable group, defined over $M$. We make heavy use of the structure theorem for $G$ given in Section 2 of 
\cite{AnnalisaAnand}: there is a definable (over $M$) normal subgroup $H$ of $G$ such that 
\newline
(i) $H$ is definably connected (solvable) and ``torsion-free''
\newline
(ii) $G/H$ is definably compact, so $\fsg$ by  \cite{AnnalisaAnand}. We will denote $G/H$ by $T$ (hopefully without ambiguity) even though $T$ might be noncommutative. 

By Proposition 4.7 of \cite{AnnalisaAnand}, there is a global, left $H$-invariant type of $H$, definable over $M$. 

We now let $M_{0} = M^{\ext}$ and pass to $Th(M_{0})$. By Theorem \ref{thm: fsg and definable generics}, $G/H = T$ remains $\fsg$
and there is still a global $H$ invariant type of $H$, definable over $M_{0}$. In particular $H$ is definably 
extremely amenable, so $S_{H}(M_{0})$ has a fixed point under the action of $H(M_{0})$, hence the unique minimal 
definable $H(M_{0})$-flow is trivial, and {\em every} definable action of $H(M_{0})$ on a compact space has a fixed 
point. Now the surjective homomorphism $\pi:G\to T$ induces a surjective continuous function $\pi:S_{G}(M_{0})\to 
S_{T}(M_{0})$, which is clearly also a semigroup homomorphism. Let ${\mathcal M}(G)$ be some minimal subflow of 
$S_{G}(M_{0})$. So $\pi({\mathcal M}(G)) = {\mathcal M}(T)$, the unique minimal subflow of $S_{T}(M_{0})$ (which is the set of generic types by Section \ref{sec: Ellis for fsg}). The main point is:

\begin{lem} The restriction of $\pi$ to ${\mathcal M}(G)$ is a homeomorphism with ${\mathcal M}(T)$. 

\end{lem}

\begin{proof} This is a rather general topological dynamics fact (under the hypotheses), surely with a reference somewhere, but we give a proof nevertheless. 

The action of $G(M_{0})$ on ${\mathcal M}(G)$, induces an action (definable) of $H(M_{0})$ on ${\mathcal M}(G)$ which as remarked above (definable extreme amenability of $H(M_{0})$) has a fixed point, which we call $p$. So note that ${\mathcal M}(G) = \overline{G(M_{0})p} = \{q\cdot p: q\in S_{G}(M_{0})\}$. 
As $H(M_{0})$ fixes $p$ it follows by definability of $p$ that $H(M')$ fixes ${\bar p}$ where $M'$ is a saturated model extending $M$, and ${\bar p}$ the unique heir of $p$ over $M'$.  So if $g\in G(M')$ then $g{\bar p}$ depends only on the coset $gH$. So for  $c\in T(M')$, $c{\bar p}$ is well-defined (as $g\bar p$ for some/any $g\in G(M')$ such that $gH = c$). Hence we can define $q\cdot p$ for $q\in S_T(M_{0})$ as $c{\bar p}|M_{0}$, namely $\tp(ca/M_{0})$ where $c\in T(M')$ realizes $q$ and $a$ realizes ${\bar p}$, and we can easily check that
\newline
(i) ${\mathcal M}(G) = \{q\cdot p: q\in S_{T}(M_{0})\}$,
\newline
(ii) $T(M_{0})$ acts on ${\mathcal M}(G)$, by $c(q\cdot p) = cq\cdot p$.
\newline
(iii) Under this action  ${\mathcal M}(G)$ is a minimal definable $T(M_{0})$-flow. 
\newline
(iv)  $\pi|{\mathcal M}(G)$ is a map of $T(M_{0})$-flows.

\vspace{2mm}

It follows from (iii) and (iv) that $\pi|{\mathcal M}(G)$ is a homeomorphism with ${\mathcal M}(T)$, as ${\mathcal M}(T)$ is the universal minimal definable $T(M_{0})$-flow. (By universality we also have a $T(M_{0})$-flow map, $f:{\mathcal M}(T) \to {\mathcal M}(G)$. The composition of $f$ with $\pi$ has to be an automorphism of ${\mathcal M}(T)$, using  Claim (iv) preceding Proposition 3.12 of \cite{GisPenPil}). So the lemma is proved. 
\end{proof}

By the Lemma and the fact that $\pi$ is a semigroup homomorphism, $\pi$ induces an isomorphism between Ellis groups $u{\mathcal M}(G)$ and $\pi(u){\mathcal M}(T)$.  The canonical map $\pi(u){\mathcal M}(T)  \to  T/T^{00}$ is an isomorphism by part \ref{thm: EllisFSG} of Theorem \ref{thm: Ellis group conjectures}. On the other hand as $H = H^{00}$ (by Proposition \ref{prop: def ext am iff}), it follows that $H < G^{00}$ and hence the map $G\to T$ induces an isomorphism of $G/G^{00}$ with $T/T^{00}$. So clearly the canonical homomorphism from $u{\mathcal M}(G)$ to $G/G^{00}$ is an isomorphism, as required. 

\bibliographystyle{alpha}
\bibliography{common}

\def\cprime{$'$}
\begin{thebibliography}{GPP12b}

\bibitem[Aus88]{Auslander}
Joseph Auslander.
\newblock {\em Minimal flows and their extensions}, volume 153 of {\em
  North-Holland Mathematics Studies}.
\newblock North-Holland Publishing Co., Amsterdam, 1988.
\newblock Notas de Matem{\'a}tica [Mathematical Notes], 122.

\bibitem[CK12]{CheKap}
Artem Chernikov and Itay Kaplan.
\newblock Forking and dividing in ${NTP}_2$ theories.
\newblock {\em J. Symbolic Logic}, 77(1):1--20, 2012.

\bibitem[CP12]{AnnalisaAnand}
Annalisa Conversano and Anand Pillay.
\newblock Connected components of definable groups and {$o$}-minimality {I}.
\newblock {\em Adv. Math.}, 231(2):605--623, 2012.

\bibitem[CS]{2012arXiv1202.2650C}
Artem {Chernikov} and Pierre {Simon}.
\newblock {Externally definable sets and dependent pairs II}.
\newblock {\em Transactions of the American Mathematical Society, accepted
  (arXiv:1202.2650)}.

\bibitem[CS12]{ExtDefI}
Artem Chernikov and Pierre Simon.
\newblock Externally definable sets and dependent pairs.
\newblock {\em Israel Journal of Mathematics}, pages 1--17, 2012.

\bibitem[Gis11]{Jakub}
Jakub Gismatullin.
\newblock Model theoretic connected components of groups.
\newblock {\em Israel J. Math.}, 184:251--274, 2011.

\bibitem[GPP12a]{GisPenPilSL2R}
J.~Gismatullin, D.~Penazzi, and A.~Pillay.
\newblock Some model theory of {$SL(2,R)$}.
\newblock {\em Preprint, arXiv:1208.0196}, 2012.

\bibitem[GPP12b]{GisPenPil}
Jakub Gismatullin, Davide Penazzi, and Anand Pillay.
\newblock On compactifications and the topological dynamics of definable
  groups.
\newblock {\em Preprint, arXiv:1212.3176}, 2012.

\bibitem[HP11]{NIP2}
Ehud Hrushovski and Anand Pillay.
\newblock On {NIP} and invariant measures.
\newblock {\em J. Eur. Math. Soc. (JEMS)}, 13(4):1005--1061, 2011.

\bibitem[HPP08]{NIP1}
Ehud Hrushovski, Ya'acov Peterzil, and Anand Pillay.
\newblock Groups, measures, and the {NIP}.
\newblock {\em J. Amer. Math. Soc.}, 21(2):563--596, 2008.

\bibitem[HPS13]{HruPilSimMeas}
Ehud Hrushovski, Anand Pillay, and Pierre Simon.
\newblock Generically stable and smooth measures in {NIP} theories.
\newblock {\em Trans. Amer. Math. Soc.}, 365(5):2341--2366, 2013.

\bibitem[Hru]{HrushMetastable}
Ehud Hrushovski.
\newblock Valued fields and metastable groups.
\newblock {\em Preprint}.

\bibitem[New09]{New4}
Ludomir Newelski.
\newblock Topological dynamics of definable group actions.
\newblock {\em J. Symbolic Logic}, 74(1):50--72, 2009.

\bibitem[Pil13]{AnandTopDyn}
Anand Pillay.
\newblock Topological dynamics and definable groups.
\newblock {\em J. Symbolic Logic}, 78:657--666, 2013.

\bibitem[Poi01]{PoizatStableGroups}
Bruno Poizat.
\newblock {\em Stable groups}, volume~87 of {\em Mathematical Surveys and
  Monographs}.
\newblock American Mathematical Society, Providence, RI, 2001.
\newblock Translated from the 1987 French original by Moses Gabriel Klein.

\bibitem[{She}07]{Sh886}
Saharon {Shelah}.
\newblock {Definable groups for dependent and 2-dependent theories}.
\newblock {\em Preprint, arXiv:math/0703045}, March 2007.

\bibitem[She08]{MR2361885}
Saharon Shelah.
\newblock Minimal bounded index subgroup for dependent theories.
\newblock {\em Proc. Amer. Math. Soc.}, 136(3):1087--1091 (electronic), 2008.

\bibitem[She09]{ShelahDependentCont}
Saharon Shelah.
\newblock Dependent first order theories, continued.
\newblock {\em Israel J. Math.}, 173:1--60, 2009.

\bibitem[Sim12]{SimonDPmin}
Pierre Simon.
\newblock Invariant types in dp-minimal theories.
\newblock {\em Preprint, arXiv:1210.4479}, 2012.

\bibitem[Sim13]{SimContraction}
Pierre Simon.
\newblock More on invariant types in {NIP} theories.
\newblock {\em Note, available on authors webpage}, 2013.

\end{thebibliography}

\end{document}